\documentclass[a4paper,11pt]{amsart}
\usepackage{amssymb,amsfonts,amsxtra,amscd,mathrsfs,stmaryrd}
\usepackage[all]{xy}
\usepackage{fullpage}
\usepackage{graphicx}
\theoremstyle{plain}
\newtheorem{theorem}{Theorem}[section]
\newtheorem{lemma}[theorem]{Lemma}
\newtheorem{cor}[theorem]{Corollary}
\newtheorem{prop}[theorem]{Proposition}
\theoremstyle{definition}
\newtheorem{defi}[theorem]{Definition}

\theoremstyle{remark}
\newtheorem{rem}[theorem]{Remark}
\numberwithin{equation}{section}
\newcommand{\gf}{\ensuremath{\mathbb{K}}}

\newcommand{\invlim}[2]{\ensuremath{\varprojlim_{#1} #2}}
\newcommand{\innprod}{\ensuremath{\langle -,- \rangle}}
\newcommand{\cotimes}{\ensuremath{\hat{\otimes}}}
\newcommand{\noproof}{\begin{flushright} \ensuremath{\square} \end{flushright}}
\newcommand{\eext}[1]{\ensuremath{E_{\mathrm{ext}}(#1)}}
\newcommand{\eint}[1]{\ensuremath{E_{\mathrm{int}}(#1)}}
\newcommand{\vext}[1]{\ensuremath{V_{\mathrm{ext}}(#1)}}
\newcommand{\vint}[1]{\ensuremath{V_{\mathrm{int}}(#1)}}
\newcommand{\chord}[1]{\ensuremath{\mathcal{C}(#1)}}
\DeclareMathOperator{\id}{id}

\DeclareMathOperator{\Der}{Der}
\DeclareMathOperator{\Aut}{Aut}

\DeclareMathOperator{\Bij}{Bij}

\begin{document}
\title{A nondiagrammatic description of the Connes-Kreimer Hopf algebra}
\author{Alastair Hamilton}
\address{Department of Mathematics and Statistics, Texas Tech University, Lubbock, TX 79409-1042. USA.} \email{alastair.hamilton@ttu.edu}
\begin{abstract}
We demonstrate that the fundamental algebraic structure underlying the Connes-Kreimer Hopf algebra -- the insertion pre-Lie structure on graphs -- corresponds directly to the canonical pre-Lie structure of polynomial vector fields. Using this fact, we construct a Hopf algebra built from tensors that is isomorphic to a version of the Connes-Kreimer Hopf algebra that first appeared in the perturbative renormalization of quantum field theories.
\end{abstract}
%\thanks{}
\maketitle
%\tableofcontents

%
%
%

\section{Introduction}

\subsection{Background}

In this paper we provide a nondiagrammatic description of the Connes-Kreimer Hopf algebra introduced in \cite{CKhopf} that underlies the perturbative renormalization of quantum field theories. This Hopf algebra was constructed by Connes and Kreimer through the use of a certain fundamental algebraic structure on graphs \cite{CKelim}; a pre-Lie structure in which the pre-Lie bracket of two graphs is described by inserting one graph into the other. The Connes-Kreimer Hopf algebra is then realized as the universal enveloping algebra of the resulting Lie algebra. Connes and Kreimer demonstrated that their Hopf algebra encoded the complicated graphical combinatorics present in the BPHZ renormalization algorithm \cite{bogpar}, \cite{hepp}, \cite{zimmer}; in that the Feynman amplitudes of the corresponding field theory were realized as a loop of characters in the Hopf algebra and the renormalized values of these Feynman amplitudes were obtained through the Birkhoff factorization of this loop.

We should be careful to explain what we mean by a nondiagrammatic description of this Hopf algebra, as one may consider Connes and Kreimer's original description to be nondiagrammatic in the sense that it formulated the problem of renormalization through the Birkhoff factorization. Here we are referring to the fact that the Connes-Kreimer Hopf algebra is constructed from combinatorial objects, namely graphs. In this paper we construct a Hopf algebra which is isomorphic to a version of the Connes-Kreimer Hopf algebra and that is built from tensors rather than graphs. Here we draw upon the work of Kontsevich \cite{kont} in which he showed how to describe graph complexes in terms of tensors using the invariant theory of the symplectic linear group. Using this correspondence, we demonstrate that the insertion pre-Lie structure on graphs corresponds directly to the canonical pre-Lie structure on polynomial vector fields. This fact is then used to construct our Hopf algebra.

Let us take a moment to point out some distinctions between the algebraic structures that we consider in this paper and those considered in \cite{CKhopf}. Firstly, the most significant distinction is that our graphs are not decorated by external parameters such as space-time points or momenta. This is not so much because these structures are incompatible with the mathematical framework that we present here, but rather because including them reduces the level of generality that we work in and introduces another level of technicalities. Such decorations will in general depend on the quantum field theory that one is considering, whereas here we prefer to focus simply on the underlying algebraic structures. In \cite{CKhopf} Connes and Kreimer construct their Hopf algebra for $\phi^3$ theory in six dimensions, although they point out that their results can be extended to any renormalizable quantum field theory. In some sense, it would be nice to have some universal Hopf algebra whose graphs are not decorated by external parameters, whereby the external parameters appear for any quantum field theory as part of the representation of the Hopf algebra, and in which the problem of renormalization is formulated through the factorization of characters as above. However, the formula (cf. Equation (6) of \cite{CKhopf}) for the diagonal of the Connes-Kreimer Hopf algebra seems to twist the graphs and their external structures together, making this proposal appear nontrivial to realize. If it were not for this fact, the difference between graphs with or without external parameters would not be material.

Secondly, the Hopf algebra we consider here is generated by all graphs and not just those that are one particle irreducible. This is due to the fact that there seems to be no a priori way to pick out one particle irreducible graphs in our mathematical framework of tensors. This distinction however is not material to the subject of renormalization. The reason for only considering one particle irreducible graphs is that they are the only graphs that require counterterms (cf. Equation (5.5.3) of \cite{collins}). One may just as easily formulate the BPHZ renormalization algorithm for all graphs, cf. \cite[\S 5.3]{collins}.

Finally, as mentioned above, in this paper we concern ourselves only with the description of the algebraic structures involved and not with the concomitant matters regarding the renormalization of certain quantum field theories. Partly, this is because we prefer to focus on the connection between the insertion pre-Lie structure on graphs and the canonical pre-Lie structure of polynomial vector fields. Partly, it is also because, for the reasons outlined above, making this connection is not completely straightforward. It is known \cite[\S 5.6]{collins}, \cite{costeff} that one may formulate the process of renormalization without resorting to the graphical combinatorics of the BPHZ algorithm. In these settings, counterterms are computed iteratively and subtracted from the overall Lagrangian. One might hope that the nondiagrammatic descriptions of the algebraic structures of Connes and Kreimer that we provide here may provide a way to bring their methods and ideas to this setting.

Throughout the paper we work with formal objects. For example, our tendency is to work with algebras of power series rather than polynomials. This tendency arises due to the way in which the Connes-Kreimer Hopf algebra is constructed as the dual of a commutative noncocommutative Hopf algebra. This leads to the consideration of objects which have a natural inverse limit structure and a corresponding topology. In this sense, we often consider Hopf algebras in the category of profinite vector spaces, which are a formally complete version of ordinary Hopf algebras and are in one-to-one correspondence with them in the sense that one notion is dual to the other. This distinction causes few real problems beyond some irritating technicalities concerning the convergence of certain expressions involving infinite summations, however it may be unfamiliar to some readers. Consequently, we provide a brief description of the necessary background in Section \ref{sec_notcon} that follows.

The breakdown of the paper is as follows. In Section \ref{sec_graphical} we recall the original description \cite{CKhopf} of the Connes-Kreimer Hopf algebra in terms of graphs and describe how, by virtue of a Milnor-Moore type theorem, this Hopf algebra may be described in terms of a pre-Lie structure on connected graphs in which one graph is inserted into another. In Section \ref{sec_nongraphical} we proceed to give a nondiagrammatic description of this Hopf algebra. Here the invariant theory of the orthogonal groups plays a crucial role in allowing us to pass from one setting to another. It is in this section that we formulate and prove our main theorem describing the insertion pre-Lie structure of Connes and Kreimer in terms of the canonical pre-Lie structure on polynomial vector fields.

\subsection{Notation and conventions} \label{sec_notcon}

A vector space $V$ is \emph{profinite} if it is an inverse limit
\[ V=\invlim{\alpha\in\mathcal{I}}{V_\alpha} \]
of finite-dimensional vector spaces $V_\alpha$. The presentation of $V$ as an inverse limit induces the inverse limit topology upon $V$. Profinite vector spaces form a category $\mathcal{P}\mathrm{Vect}$ in which the morphisms are \emph{continuous} linear maps. This category is anti-equivalent to the usual category of all vector spaces under the functors
\begin{equation} \label{eqn_cateqv}
\begin{array}{ccc}
\mathrm{Vect} & \rightleftharpoons & \mathcal{P}\mathrm{Vect} \\
V & \rightharpoonup & V^* \\
W^{\dag} & \leftharpoondown & W
\end{array}
\end{equation}
where $V^*$ denotes the linear dual of $V$ and $W^{\dag}$ denotes the \emph{continuous} linear dual of $W$.

If $V=\invlim{\alpha\in\mathcal{I}}{V_\alpha}$ and $U=\invlim{\beta\in\mathcal{J}}{U_\beta}$ are two profinite vector spaces, we define their \emph{completed tensor product} $\cotimes$ by
\[ V\cotimes U:=\invlim{\alpha,\beta\in\mathcal{I}\times\mathcal{J}} V_\alpha\otimes U_\beta. \]
This construction is functorial in $V$ and $U$. There is a canonical map
\begin{equation} \label{eqn_tensorinc}
V\otimes U \hookrightarrow V\cotimes U
\end{equation}
defined by the universal property of inverse limits, whose image is dense in $V\cotimes U$.

The equivalence \eqref{eqn_cateqv} of categories identifies the tensor product $\otimes$ in $\mathrm{Vect}$ with the completed tensor product $\cotimes$ in $\mathcal{P}\mathrm{Vect}$ and vice versa. Hence a monoid in one category is the same thing as a comonoid in the other category. We use this perspective tacitly throughout the paper. In particular, if $X$ is a Hopf algebra in the usual sense, then its dual $X^*$ is a Hopf algebra in the category of profinite vector spaces. Conversely, if $Y$ is a Hopf algebra in the category of profinite vector spaces, then its continuous dual $Y^{\dag}$ is a Hopf algebra in the usual sense. Furthermore, one may show that just as $\ast$ turns direct sums into direct products, $\dag$ turns direct products into direct sums. Similarly, just as $\otimes$ is distributive with respect to direct sums, the completed tensor product $\cotimes$ is distributive with respect to direct products.

Given a profinite vector space $V$, we define the completed tensor algebra of $V$ by
\[ \widehat{T}(V):=\prod_{n=0}^\infty V^{\cotimes n}. \]
Likewise, we define the completed symmetric algebra of $V$ by
\[ \widehat{S}(V):=\prod_{n=0}^\infty V^{\cotimes n}/\mathbb{S}_n, \]
where the $\mathbb{S}_n$ action is the unique action that continuously extends the usual action of $\mathbb{S}_n$ on $V^{\otimes n}$. Just as the usual tensor and symmetric algebras have the structure of Hopf algebras, the completed tensor and symmetric algebras have the structure of cocommutative Hopf algebras in the category of profinite vector spaces; that is to say we have maps
\begin{displaymath}
\begin{split}
\hat{\mu}: & \widehat{T}(V)\cotimes\widehat{T}(V)\to\widehat{T}(V) \\
\hat{\nabla}: & \widehat{T}(V)\to\widehat{T}(V)\cotimes\widehat{T}(V) \\
\hat{S}: & \widehat{T}(V)\to\widehat{T}(V)
\end{split}
\end{displaymath}
which are the unique continuous extensions of the usual maps that make $T(V)$ into a Hopf algebra, and likewise for the symmetric algebra $\widehat{S}(V)$.

Suppose that we have the structure of a Lie algebra
\[ [-,-]:\mathfrak{g}\cotimes\mathfrak{g}\to\mathfrak{g} \]
on a profinite vector space $\mathfrak{g}$. We define the completed universal enveloping algebra $\widehat{U}(\mathfrak{g})$ of $\mathfrak{g}$ as the quotient of the tensor algebra $\widehat{T}(\mathfrak{g})$ by the \emph{closure} of the subspace generated by the usual relations defining $U(\mathfrak{g})$. It is straightforward to check that the Hopf algebra structure defined above on $\widehat{T}(\mathfrak{g})$ descends to the completed universal enveloping algebra $\widehat{U}(\mathfrak{g})$.

Throughout the paper we work over a ground field $\gf$ of characteristic zero. For every positive integer $n$, there is a canonical vector space $V_n:=\gf^n$ of dimension $n$ equipped with a nondegenerate symmetric bilinear form. If $x_1,\ldots,x_n$ is the standard basis of $V_n$, then
\begin{equation} \label{eqn_canonicalform}
\langle x_i,x_j \rangle := \left\{\begin{array}{ll} 1, & i=j \\ 0, & i\neq j \end{array}\right\}.
\end{equation}
If $V$ is a vector space with a symmetric nondegenerate bilinear form $\innprod$, then we denote the group of endomorphisms $\phi$ of $V$ satisfying
\[ \langle \phi(v),\phi(w) \rangle = \langle v,w \rangle, \text{ for all }v,w \in V \]
by $O(V)$. Such transformations will be referred to as \emph{orthogonal} transformations. In particular, we will denote $O(V_n)$ by $O(n)$.

Given a symmetric bilinear form $\innprod$ on a vector space $V$, the inverse form $\innprod^{-1}$ on $V^*$ is defined by the following commutative diagram;
\[ \xymatrix{ & \gf \\ V\otimes V \ar[ur]^{\innprod} \ar[rr]^{D \otimes D} && V^*\otimes V^* \ar[ul]_{\innprod^{-1}}} \]
where $D(v):=[x\mapsto\langle v,x \rangle]$.

\section{The Connes-Kreimer Hopf algebra and its graphical description} \label{sec_graphical}

In this section we recall the graphical description of the Connes-Kreimer Hopf algebra as it was outlined in \cite{CKhopf}. We begin by introducing the elementary objects, namely graphs, that it is constructed from and describe some elementary operations on them. We then proceed to a description of the Connes-Kreimer Hopf algebra and its primitive elements. After formulating a Milnor-Moore type theorem for this Hopf algebra, we recall the fundamental insertion pre-Lie structure on graphs defined in \cite{CKhopf} and \cite{CKelim} that describes the Lie algebra structure of its primitive elements.

\subsection{Graphs}

We start by introducing elementary definitions of the notion of graph and subgraph and describe how to collapse a subgraph of a graph to a point.

\begin{defi} \label{def_graph}
A graph $\Gamma$ is a set $\Gamma$ consisting of the \emph{half-edges} of the graph together with the data of:
\begin{enumerate}
\item \label{item_edges}
A partition $E(\Gamma)$ of $\Gamma$ into pairs, called the set of \emph{edges} of $\Gamma$.
\item
A partition $V(\Gamma)$ of $\Gamma$, called the set of \emph{vertices} of $\Gamma$. The cardinality of a vertex $v\in V(\Gamma)$ is called its \emph{valency}. Vertices may have any valency.
\item \label{item_external}
A subset $\vext{\Gamma}$ of the 1-valent vertices of $\Gamma$, called the \emph{external vertices}.
\end{enumerate}

\begin{figure}[htp]
\centering
\includegraphics{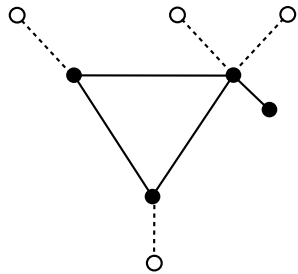}
\caption{Here, the external vertices are denoted by white dots and the external edges are denoted by dashed lines. Internal vertices are denoted by solid dots and internal edges are denoted by solid lines.}
\end{figure}

Any vertex which is not an external vertex is called an \emph{internal vertex}. The set of internal vertices of $\Gamma$ will be denoted by $\vint{\Gamma}$. An edge will be called \emph{external} if it is connected to an external vertex; otherwise, it will be called \emph{internal}. The set of external edges will be denoted by $\eext{\Gamma}$ and the set of internal edges will be denoted by $\eint{\Gamma}$. The empty graph is considered to be a graph and denoted by $\mathbf{1}$. We say that two graphs are isomorphic if there is a bijective mapping of the half-edges from one to the other preserving structures \eqref{item_edges} to \eqref{item_external} above.
\end{defi}

We may contract any internal edge of a graph $\Gamma$.

\begin{defi}
Let $e=\{h_1,h_2\}$ be an internal edge of a graph $\Gamma$. We define $\Gamma/e$ to be the result of contracting the edge $e$ in $\Gamma$. Specifically:
\begin{enumerate}
\item
\[ E(\Gamma/e)=E(\Gamma)-\{e\}. \]
\item
\begin{enumerate}
\item
If $e$ is not a loop and hence is incident to two vertices
\[v_1=\{h_1,x_1,\ldots,x_k\} \quad\text{and}\quad v_2=\{h_2,y_1,\ldots,y_l\}\]
then the vertices of $\Gamma/e$ consist of all the other vertices of $\Gamma$ together with a new vertex
\[ v=\{x_1,\ldots,x_k,y_1,\ldots,y_l\} \]
formed by combining $v_1$ and $v_2$.
\item
If $e$ is a loop incident to a vertex
\[ v=\{h_1,h_2,x_1,\ldots,x_k\} \]
then the vertices of $\Gamma/e$ consist of all the other vertices of $\Gamma$ together with the new vertex
\[ v'=\{x_1,\ldots,x_k\} \]
formed by removing the loop $e$.
\end{enumerate}
\item
The external vertices of $\Gamma/e$ are the same as those of $\Gamma$.
\end{enumerate}
\end{defi}

\begin{figure}[htp]
\centering
\includegraphics{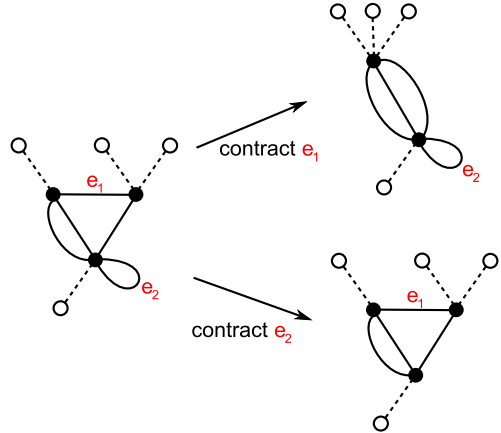}
\caption{Contracting edges and loops in a graph.}
\end{figure}

\begin{rem} \label{rem_contract}
Note that one may check that it does not matter what order one contracts edges in a graph, that is to say that
\[ (\Gamma/e_1)/e_2 = (\Gamma/e_2)/e_1 \]
for any two internal edges $e_1$ and $e_2$.
\end{rem}

\begin{defi}
A \emph{subgraph} $\gamma$ of a graph $\Gamma$ is a subset of the internal edges $\eint{\Gamma}$ of $\Gamma$. By Remark \ref{rem_contract}, we may associate to any subgraph $\gamma=\{e_1,\ldots,e_n\}$ of $\Gamma$ the graph
\[ \Gamma/\gamma:=(\ldots((\Gamma/e_1)/e_2)\ldots/e_n)  \]
with all the edges of $\gamma$ contracted.

To any subgraph $\gamma=\{e_1,\ldots,e_n\}$ of $\Gamma$, we may associate a graph in the sense of Definition \ref{def_graph}, which, by an abuse of notation, we shall also denote by $\gamma$. It is constructed as follows. Let
\[ \{v\in V(\Gamma): v\cap\left(\cup_{i=1}^n e_i\right) \neq \phi\}=\{v_1,\ldots,v_r\} \]
be a list of those vertices which intersect the subgraph $\gamma$ and consider the set
\[ H=\bigcup_{i=1}^r v_i \]
of half-edges of $\Gamma$ formed by their union. Let $h_1,\ldots,h_k$ be a list of those half-edges in $H$ which do not belong to any edge $e_i$ of $\gamma$;
\[ H-\bigcup_{i=1}^n e_i = \{h_1,\ldots,h_k\}. \]
These half-edges will form the ends of the external edges of $\gamma$. To this end we introduce another disjoint set of half-edges $h'_1,\ldots,h'_k$ to form the other ends of these external edges. The graph $\gamma$ may now be defined as follows:

\begin{figure}[htp]
\centering
\includegraphics{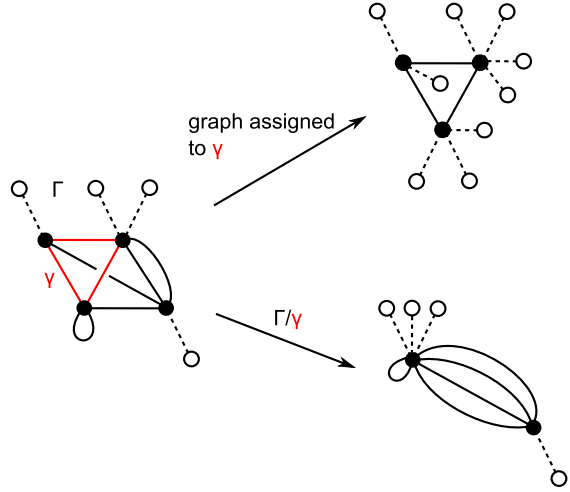}
\caption{Constructions associated to subgraphs.}
\end{figure}

\begin{enumerate}
\item
The half-edges of $\gamma$ consist of $H\cup\{h'_1,\ldots,h'_k\}$.
\item
The edges of $\gamma$ are
\[ E(\gamma)=\{e_1,\ldots,e_n,\{h_1,h'_1\},\ldots\{h_k,h'_k\}\}. \]
\item
The vertices of $\gamma$ are
\[ V(\gamma)= \{v_1,\ldots,v_r,\{h'_1\},\ldots,\{h'_k\}\}. \]
The external vertices of $\gamma$ are $\vext{\gamma}=\{\{h'_1\},\ldots,\{h'_k\}\}$.
\end{enumerate}
\end{defi}

\subsection{The Connes-Kreimer Hopf algebra} \label{sec_ckhopf}

Now that we have introduced some of the elementary objects, we are in a position to begin recalling the definition of the Connes-Kreimer Hopf algebra. We begin by describing a simple commutative algebra structure on graphs.

\begin{defi}
Let $\mathcal{H}$ be the vector space freely generated by isomorphism classes of graphs. Given two graphs $\Gamma_1$ and $\Gamma_2$ we may form their disjoint union $\Gamma_1\sqcup\Gamma_2$. This gives $\mathcal{H}$ the structure of an algebra, with multiplication $\mu$ given by disjoint union. We will denote the subspace of $\mathcal{H}$ generated by \emph{connected} graphs by $\mathcal{H}_c$. Note that the empty graph is not considered to be connected. The spaces $\mathcal{H}$ and $\mathcal{H}_c$ posses natural gradings
\[ \mathcal{H}=\bigoplus_{n=0}^\infty \mathcal{H}_n \qquad \mathcal{H}_c=\bigoplus_{n=0}^\infty \mathcal{H}_{c,n} \]
where $\mathcal{H}_n$ and $\mathcal{H}_{c,n}$ are those subspaces of $\mathcal{H}$ and $\mathcal{H}_c$ respectively which are generated by graphs with precisely $n$ edges. These gradings are of finite-type in that each of the subspaces $\mathcal{H}_n$ and $\mathcal{H}_{c,n}$ are finite-dimensional.

Occasionally, we will wish to use a finer grading than the one above. We may write
\begin{equation} \label{eqn_graphgrade}
\mathcal{H}=\bigoplus_{n,m,k=0}^\infty \mathcal{H}_{n,m,k} \qquad \mathcal{H}_c=\bigoplus_{n,m,k=0}^\infty \mathcal{H}_{c;n,m,k}
\end{equation}
where both $\mathcal{H}_{n,m,k}$ and $\mathcal{H}_{c;n,m,k}$ are spanned by graphs with precisely $n$ edges, $m$ internal edges and $k$ external vertices. The subscript $\geq a$ will be used to denote the subspace spanned by elements of degree greater than or equal to $a$, for instance
\[ \mathcal{H}_{\bullet,\geq a,\bullet}:= \bigoplus_{\begin{subarray}{c} n,k=0 \\ m=a \end{subarray}}^\infty \mathcal{H}_{n,m,k} \]
will denote the subspace of $\mathcal{H}$ spanned by graphs with at least $a$ internal edges.
\end{defi}

\begin{prop} \label{prop_polygraph}
$\mathcal{H}$ is a polynomial algebra over the subspace of connected graphs, that is to say we have the following canonical isomorphism of algebras;
\begin{displaymath}
\begin{array}{ccc}
S(\mathcal{H}_c) & \cong & \mathcal{H} \\
\Gamma_1\otimes\cdots\otimes\Gamma_k & \mapsto & \Gamma_1\sqcup\cdots\sqcup\Gamma_k
\end{array}
\end{displaymath}
\end{prop}

\begin{proof}
This follows from the fact that any graph is the disjoint union of its connected components.
\end{proof}

\begin{rem} \label{rem_connprim}
Since $S(\mathcal{H}_c)$ has a canonical cocommutative comultiplication on it in which the subspace $\mathcal{H}_c$ coincides with the space of primitive elements, the above theorem also endows $\mathcal{H}$ with a corresponding cocommutative comultiplication. This comultiplication differs from the noncocommutative comultiplication of Connes and Kreimer that we will introduce later in Definition \ref{def_ckcoprod}.
\end{rem}

The Connes-Kreimer Hopf algebra will be constructed as the dual of $\mathcal{H}$. To assist us in describing it, we introduce a nondegenerate bilinear form on graphs. Since $\mathcal{H}$ has a preferred basis consisting of isomorphism classes of graphs, this allows us to define a symmetric bilinear form on $\mathcal{H}$ as follows.

\begin{defi} \label{def_graphform}
Given two graphs $\Gamma_1$ and $\Gamma_2$ we define
\begin{equation} \label{eqn_graphform}
\langle\Gamma_1,\Gamma_2\rangle = \left\{ \begin{array}{rl} 1, & \Gamma_1\cong\Gamma_2 \\ 0, & \Gamma_1\ncong\Gamma_2\end{array} \right\}.
\end{equation}
This defines a symmetric bilinear form on $\mathcal{H}$ which is nondegenerate in the sense that it is nondegenerate on each $\mathcal{H}_n$. Obviously the pairing is trivial in the event that $\Gamma_1$ and $\Gamma_2$ have a different number of edges so that $\mathcal{H}_i$ and $\mathcal{H}_j$ are orthogonal if $i\neq j$.
\end{defi}

\begin{defi}
Let us denote by $\widehat{\mathcal{H}}$ the vector space freely generated by \emph{formal} infinite linear combinations of isomorphism classes of graphs. This space has the same natural grading as $\mathcal{H}$;
\[ \widehat{\mathcal{H}} = \prod_{n=0}^\infty \mathcal{H}_n. \]
\end{defi}

As before, we have the identification $\widehat{\mathcal{H}}\cong\widehat{S}(\widehat{\mathcal{H}}_c)$ of Proposition \ref{prop_polygraph}. The bilinear form \eqref{eqn_graphform} defined in Definition \ref{def_graphform} yields the following simple proposition describing this space as the space that is dual to $\mathcal{H}$.

\begin{prop} \label{prop_dualiso}
There is a canonical isomorphism between $\widehat{\mathcal{H}}$ and $\mathcal{H}^*$ given by
\begin{equation} \label{eqn_dualiso}
\begin{array}{rcl}
\widehat{\mathcal{H}}=\prod_{n=0}^\infty\mathcal{H}_n & \cong & \prod_{n=0}^\infty\mathcal{H}_n^*=\mathcal{H}^* \\
\Gamma & \mapsto & |\Aut(\Gamma)|\langle\Gamma,-\rangle
\end{array}
\end{equation}
\end{prop}
\noproof

Now we can recall how the Connes-Kreimer coproduct on graphs is defined. This coproduct is the fundamental algebraic structure that encodes the graphical combinatorics of the BPHZ algorithm \cite{CKhopf}.

\begin{defi} \label{def_ckcoprod}
By Proposition \ref{prop_polygraph}, in order to define a bialgebra structure on $\mathcal{H}$, it is sufficient to specify how the comultiplication
\[ \nabla:\mathcal{H}\to\mathcal{H}\otimes\mathcal{H} \]
behaves on \emph{connected} graphs. This is defined by Connes-Kreimer \cite{CKhopf} using the following formula;
\begin{equation} \label{eqn_ckcoprod}
\nabla(\Gamma):= \mathbf{1}\otimes\Gamma + \Gamma\otimes\mathbf{1} + \sum_{\gamma\subsetneq\Gamma} \gamma\otimes\Gamma/\gamma
\end{equation}
where the sum is taken over all nontrivial subgraphs of $\Gamma$.
\end{defi}

\begin{rem} \label{rem_cograding}
Note that if $\Gamma\in\mathcal{H}_n$, then $\nabla(\Gamma)$ will be an inhomogeneous sum of elements of total degree $k$ for $k$ between $n$ and $3n$. However, if we consider the grading on $\mathcal{H}$ defined by the number of \emph{internal} edges of a graph, then $\nabla$ will have degree zero in this grading. These facts may be summarized by the identity
\[ \nabla(\mathcal{H}_{n,m,\bullet})\subset\bigoplus_{k=n}^{3n}\left(\bigoplus_{\begin{subarray}{c} i_1+j_1=k \\ i_2+j_2=m \end{subarray}} \mathcal{H}_{i_1,i_2,\bullet}\otimes\mathcal{H}_{j_1,j_2,\bullet}\right). \]
\end{rem}

Theorem 1 of \cite{CKhopf} due to Connes-Kreimer states that this comultiplication is coassociative and gives $\mathcal{H}$ the structure of a commutative Hopf algebra. The Connes-Kreimer Hopf algebra is then given by taking the linear dual of this Hopf algebra. The bilinear form introduced in Definition \ref{def_graphform} allows us to describe this dual Hopf algebra in terms of algebraic structures defined on $\widehat{\mathcal{H}}$.

\begin{defi}
The vector space $\widehat{\mathcal{H}}$ may be endowed with the structure of a cocommutative Hopf algebra (in the category $\mathcal{P}\mathrm{Vect}$ of profinite vector spaces) in the following manner. Let $I:\widehat{\mathcal{H}}\to\mathcal{H}^*$ be the map given by \eqref{eqn_dualiso}. The multiplication and comultiplication
\[ \star:\widehat{\mathcal{H}}\cotimes\widehat{\mathcal{H}}\to\widehat{\mathcal{H}} \quad\text{and}\quad \Delta:\widehat{\mathcal{H}}\to\widehat{\mathcal{H}}\cotimes\widehat{\mathcal{H}} \]
are defined by the formulae:
\begin{displaymath}
\begin{split}
I\circ\star & = \nabla^*\circ(I\cotimes I) \\
\mu^*\circ I & = (I\cotimes I)\circ\Delta
\end{split}
\end{displaymath}
\end{defi}

\begin{rem} \label{rem_grading}
It follows from Remark \ref{rem_cograding} that $\star$ has degree zero in the grading by the number of \emph{internal} edges. Additionally, if $\Gamma_1$ is a formal sum of elements from $\mathcal{H}_k$ for $k\geq 3n$ then both $\Gamma_1\star\Gamma_2$ and $\Gamma_2\star\Gamma_1$ will be a formal sum of elements from $\mathcal{H}_k$ for $k\geq n$. We may summarize this as
\begin{equation} \label{eqn_grading}
\begin{split}
\star\left([\mathcal{H}_{\geq 3n,m_1,\bullet}] , [\mathcal{H}_{\bullet,m_2,\bullet}]\right) &\subset \mathcal{H}_{\geq n,m_1+m_2,\bullet} \\
\star\left([\mathcal{H}_{\bullet,m_1,\bullet}] , [\mathcal{H}_{\geq 3n,m_2,\bullet}]\right) &\subset \mathcal{H}_{\geq n,m_1+m_2,\bullet}
\end{split}
\end{equation}
\end{rem}

We begin by determining a more convenient description of the cocommutative comultiplication $\Delta$. Consider the isomorphism $T:S(\mathcal{H}_c)\to\mathcal{H}$ of Proposition \ref{prop_polygraph}. This induces an isomorphism
\[ T^*:\mathcal{H}^*\to\prod_{n=0}^\infty\left[(\mathcal{H}_c^*)^{\cotimes n}\right]^{\mathbb{S}_n}. \]
The latter space of invariants may be identified with the coinvariants $\widehat{S}(\mathcal{H}_c^*)=\prod_{n=0}^\infty\left[(\mathcal{H}_c^*)^{\cotimes n}\right]_{\mathbb{S}_n}$ under the map
\begin{displaymath}
\begin{array}{rcl}
\omega_n: \left[(\mathcal{H}_c^*)^{\cotimes n}\right]^{\mathbb{S}_n} & \to & \left[(\mathcal{H}_c^*)^{\cotimes n}\right]_{\mathbb{S}_n} \\
f & \mapsto & \frac{1}{n!} f
\end{array}
\end{displaymath}

\begin{prop}
The following diagram of isomorphisms commutes
\begin{displaymath}
\xymatrix{ \widehat{\mathcal{H}} \ar[r]^I & \mathcal{H}^* \ar[d]^{\omega\circ T^*} \\ \widehat{S}(\widehat{\mathcal{H}}_c) \ar[r]^{S(I)} \ar[u]^{\widehat{T}} & \widehat{S}(\mathcal{H}_c^*) }
\end{displaymath}
\end{prop}

\begin{proof}
The commutativity of the above diagram is equivalent to the following identity
\begin{multline*}
|\Aut(\Gamma_1\sqcup\cdots\sqcup\Gamma_n)|\langle\Gamma_1\sqcup\cdots\sqcup\Gamma_n,\gamma_1\sqcup\cdots\sqcup\gamma_n\rangle = \\ \sum_{\sigma\in\mathbb{S}_n}|\Aut(\Gamma_1)|\cdots|\Aut(\Gamma_n)|\langle\Gamma_1,\gamma_{\sigma(1)}\rangle\cdots\langle\Gamma_n,\gamma_{\sigma(n)}\rangle.
\end{multline*}
which holds for all connected graphs $\Gamma_1,\ldots\Gamma_n,\gamma_1,\ldots,\gamma_n$. This identity follows in turn from the decomposition
\[ \Aut(\Gamma_{i_1}\sqcup\cdots\sqcup\Gamma_{i_n}) = [\Aut(\Gamma_{i_1})\times\cdots\times\Aut(\Gamma_{i_n})]\rtimes\{\sigma\in\mathbb{S}_n:i_r=i_{\sigma(r)}, \text{ for all }r\} \]
of $\Aut(\Gamma_{i_1}\sqcup\cdots\sqcup\Gamma_{i_n})$ as a wreath product; where $\Gamma_i,i\in\mathcal{I}$ is a complete list of representatives for isomorphism classes of connected graphs.
\end{proof}

\begin{rem}
This proposition affords us a more convenient description of the comultiplication $\Delta$ and consequently allows us to describe the primitive elements of the Hopf algebra $(\widehat{\mathcal{H}},\star,\Delta)$. It is well-known that the isomorphism $\omega\circ T^*$ maps the comultiplication $\mu^*$ on $\mathcal{H}^*$ to the canonical cocommutative comultiplication on $\widehat{S}(\mathcal{H}_c^*)$. Consequently, the isomorphism $\widehat{T}:\widehat{S}(\widehat{\mathcal{H}}_c)\to\widehat{\mathcal{H}}$ identifies the comultiplication $\Delta$ on $\widehat{\mathcal{H}}$ with the canonical cocommutative comultiplication on $\widehat{S}(\widehat{\mathcal{H}}_c)$. From this it follows that the subspace of primitive elements of $(\widehat{\mathcal{H}},\star,\Delta)$ coincides with $\widehat{\mathcal{H}}_c$.
\end{rem}

\subsection{A Milnor-Moore type theorem}

The Hopf algebra $(\widehat{\mathcal{H}},\star,\Delta)$ may be described in terms of the universal enveloping algebra of a certain Lie algebra by virtue of a Milnor-Moore type theorem. This will allow us to analyze the algebraic structure of this Hopf algebra by analyzing the underlying structure of this Lie algebra. Our first task is to pick out the relevant subalgebra of $(\widehat{\mathcal{H}},\star,\Delta)$.

The algebra $\widehat{\mathcal{H}}$ contains a commutative subalgebra of graphs with no internal edges, which plays no role in the renormalization process and contributes nothing to the algebraic structure. Therefore, we would like to separate this subalgebra out of the discussion. Let
\[ \widehat{\mathcal{H}}_c^+:=\widehat{\mathcal{H}}_{c;\bullet,\geq 1,\bullet} \]
denote the subspace of $\widehat{\mathcal{H}}_c$ which is formally spanned by those graphs with at least one internal edge. We shall define $\widehat{\mathcal{H}}^+\cong\widehat{S}(\widehat{\mathcal{H}}_c^+)$ to be the corresponding subspace of $\widehat{\mathcal{H}}$ which is formally generated by disjoint unions of connected graphs with at least one internal edge.

\begin{prop}
The bialgebra $(\widehat{\mathcal{H}},\star,\Delta)$ is the tensor product of a strictly commutative subalgebra $K$ and the subalgebra $\widehat{\mathcal{H}}^+$
\[ \widehat{\mathcal{H}} \cong K\cotimes\widehat{\mathcal{H}}^+. \]
\end{prop}

\begin{proof}
$\widehat{\mathcal{H}}_c$ splits as a direct sum of the subspace $I:=\widehat{\mathcal{H}}_{c;\bullet 0 \bullet}$ consisting of graphs which have no internal edges and the subspace $\widehat{\mathcal{H}}_c^+$. Hence
\[ \widehat{\mathcal{H}}\cong\widehat{S}(\widehat{\mathcal{H}}_c)\cong\widehat{S}(I)\cotimes\widehat{S}(\widehat{\mathcal{H}}_c^+)=\widehat{S}(I)\cotimes\widehat{\mathcal{H}}^+. \]

Set $K:=\widehat{S}(I)$. Since it is impossible for any graph corresponding to a term in the expression \eqref{eqn_ckcoprod} for the coproduct $\nabla(\Gamma)$ to have no internal edges, unless $\Gamma$ itself has no internal edges, it follows that $K$ is not just a subalgebra of $\widehat{\mathcal{H}}$, but in fact a commutative subalgebra whose multiplication coincides with disjoint union. Furthermore, this fact also implies not only that $\widehat{\mathcal{H}}^+$ is a subalgebra of $\widehat{\mathcal{H}}$, but that
\[ \Gamma_1\star\Gamma_2 = \Gamma_1\sqcup\Gamma_2 \]
for all $\Gamma_1\in\widehat{\mathcal{H}}^+$ and $\Gamma_2\in K$. Hence we have the following isomorphism;
\begin{displaymath}
\begin{array}{ccc}
K\cotimes\widehat{\mathcal{H}}^+ & \cong & \widehat{\mathcal{H}} \\
\Gamma_1\cotimes\Gamma_2 & \mapsto & \Gamma_1\star\Gamma_2
\end{array}
\end{displaymath}
\end{proof}

Since $\widehat{\mathcal{H}}^+\cong\widehat{S}(\widehat{\mathcal{H}}_c^+)$, this induces another grading on $\widehat{\mathcal{H}}^+$ by the \emph{order} of a polynomial in $\widehat{S}(\widehat{\mathcal{H}}_c^+)$; or equivalently, the \emph{number of connected components} of a graph $\Gamma\in\widehat{\mathcal{H}}^+$. One can check that by Equation \eqref{eqn_ckcoprod},
\begin{equation} \label{eqn_stardeform}
\Gamma_1\star\Gamma_2 = \Gamma_1\sqcup\Gamma_2 + \text{terms of order } < \mathrm{order}(\Gamma_1)+\mathrm{order}(\Gamma_2).
\end{equation}
One consequence of this equation is that $\widehat{\mathcal{H}}^+$ is (formally) generated by primitive elements under $\star$, and consequently a Milnor-Moore type theorem holds

\begin{theorem} \label{thm_milmor}
The bialgebra $(\widehat{\mathcal{H}}^+,\star,\Delta)$ is canonically isomorphic to the (completed) universal enveloping algebra of its primitive elements $\widehat{\mathcal{H}}_c^+$;\
\begin{equation} \label{eqn_envelopemap}
\begin{array}{ccc}
\widehat{U}(\widehat{\mathcal{H}}_c^+) & \cong & \widehat{\mathcal{H}}^+ \\
\Gamma_1\cotimes\cdots\cotimes\Gamma_n & \mapsto & \Gamma_1\star\cdots\star\Gamma_n
\end{array}
\end{equation}
\end{theorem}

\begin{proof}
The proof involves standard arguments \cite{milnormoore}, but the situation is complicated by the fact that we are working with formal objects and the consequent need to ensure that certain expressions are convergent. Additionally, the situation is made more awkward by the fact that $\star$ is inhomogeneous in the finite-type grading by number of edges. The reader who feels no uneasiness about working with formal objects may benefit from skipping the details. We include them only for the sake of being thorough.

It follows from Remark \ref{rem_grading} and Equation \eqref{eqn_grading} that the map \eqref{eqn_envelopemap} is well-defined, in that those expressions which are implicit in its definition are convergent. More precisely, an element of $\widehat{U}(\widehat{\mathcal{H}}_c^+)$ is a formal sum of elements on the left-hand side of \eqref{eqn_envelopemap} and its image is consequently the corresponding formal sum of elements on the right-hand side of \eqref{eqn_envelopemap}, which converges in the inverse limit topology by \eqref{eqn_grading}. For this, it is important that we work with graphs with at least one internal edge.

We start by showing that \eqref{eqn_envelopemap} is surjective. Any $\Lambda\in\widehat{\mathcal{H}}^+$ will be a formal sum
\[ \Lambda=\sum_{k=1}^\infty\Lambda_k \]
of terms $\Lambda_k\in\mathcal{H}_k^+$ with $k$ edges. In turn, each $\Lambda_k$ will be a (finite) sum of terms of the form
\begin{equation} \label{eqn_conncomp}
\Gamma_1\sqcup\cdots\sqcup\Gamma_n
\end{equation}
for $n\leq k$, with each $\Gamma_i\in\mathcal{H}_c^+$. Consequently, we may write
\[ \Lambda_k=\sum_{n=1}^k\Lambda_k^n \]
where $\Lambda_k^n$ is a sum of graphs with $n$ connected components, i.e. a sum of terms of the form \eqref{eqn_conncomp}.

By Equation \eqref{eqn_stardeform} and a simple, standard inductive argument, we may show that $\Lambda_k^n$ is the image of some $\tilde{\Lambda}_k^n\in T(\mathcal{H}_c^+)$ under the map \eqref{eqn_envelopemap}, and hence $\Lambda_k$ will be the image of
\[ \tilde{\Lambda}_k:=\sum_{n=1}^k\tilde{\Lambda}_k^n \]
under the map \eqref{eqn_envelopemap}. I claim that $\Lambda$ is the image of $\tilde{\Lambda}:=\sum_{k=1}^\infty\tilde{\Lambda}_k$ under \eqref{eqn_envelopemap}, but for this to be true, we must first prove that this formal sum converges.

Consider the grading that is induced on $T(\mathcal{H}_c^+)$ by the grading on $\mathcal{H}_c^+$ whose degree coincides with the total number of edges in a graph. I claim that if $k\geq (3N)^{N}$, then $\tilde{\Lambda}_k$ will be a sum of terms of degree $\geq N$. This will prove that $\tilde{\Lambda}$ is a convergent sum and establish that \eqref{eqn_envelopemap} is surjective.

If we consider for a moment the grading induced upon $T(\mathcal{H}_c^+)$ by the number of \emph{internal} edges, then if $n\geq N$, \eqref{eqn_conncomp} will have at least $N$ internal edges, as each $\Gamma_i$ has at least one. Since $\star$ has degree zero in the grading by internal edges, it follows from the construction of $\tilde{\Lambda}_k^n$ that it will be a sum of terms of degree $\geq N$ in the grading by internal edges, and hence a sum of terms of degree $\geq N$ in the grading by the total number of edges.

Hence, we need only consider $\tilde{\Lambda}_k^n$ for $n<N$. If $n<N$ and the graph \eqref{eqn_conncomp} has $k\geq (3N)^{N}$ edges, then some connected component $\Gamma_i$ must have at least $3^{N}N^{N-1}$ edges. Hence $\Gamma_1\star\cdots\star\Gamma_n\in\mathcal{H}_{c,\geq (3N)^{N-1}}^+$
and
\[ \Gamma_1\star\cdots\star\Gamma_n = \Gamma_1\sqcup\cdots\sqcup\Gamma_n + \text{terms of order }\leq n-1. \]
Proceeding by induction to eliminate, in the same manner, terms of lower and lower order; we may conclude that $\Lambda_k^n$ is the image of an expression $\tilde{\Lambda}_k^n$ consisting of a sum of terms of degree $\geq 3N\geq N$ in the grading induced by the total number of edges.

Now that we have shown that this map is surjective, we must show that it is injective. Suppose that given $\tilde{\Lambda}\in\widehat{U}(\widehat{\mathcal{H}}_c^+)$, we know that its image $\Lambda\in\widehat{\mathcal{H}}^+$ under \eqref{eqn_envelopemap} is zero. We must show $\tilde{\Lambda}=0$. Since by Remark \ref{rem_grading}, $\star$ has degree zero in the grading induced by the number of internal edges, it is sufficient to assume that $\tilde{\Lambda}$ has homogeneous degree $n$ in this grading. In this case, $\tilde{\Lambda}$ will be represented by a (formal) sum of terms of the form
\[ \Gamma_1\otimes\cdots\otimes\Gamma_k \]
with each $\Gamma_i\in\mathcal{H}_c^+$. Note that we must have $k\leq n$ as each graph $\Gamma_i$ has at least one internal edge. Now since the number of connected components $k$ is bounded above by $n$, a standard inductive argument implies that $\tilde{\Lambda}=0$.
\end{proof}

\subsection{Pre-Lie structure on graphs}

Having reduced the study of the Hopf algebra $(\widehat{\mathcal{H}},\star,\Delta)$ to the study of the Lie algebra structure on its primitive elements, our next goal should be to describe this Lie algebra structure more explicitly. This was carried out by Connes and Kreimer in \cite{CKhopf} and \cite{CKelim} where they described this Lie algebra in terms of a pre-Lie structure on graphs in which the pre-Lie bracket of graphs was defined by inserting one graph into the other. The following definition is due to them.

\begin{defi}
Let $\Gamma_1,\Gamma_2\in\mathcal{H}^+_c$ be two connected graphs and let $v\in\vint{\Gamma_1}$ be an internal vertex of $\Gamma_1$. If the valency of $v$ differs from the number of external edges of $\Gamma_2$, we define
\[ \Gamma_1\circ_v\Gamma_2 = 0. \]
Otherwise, we will define $\Gamma_1\circ_v\Gamma_2$ to be the graph obtained by inserting $\Gamma_2$ into $\Gamma_1$ at the vertex $v$. More precisely, let $\sigma$ be a bijection between the incident half-edges of the vertex $v$ and the external edges of $\Gamma_2$.
\begin{displaymath}
\begin{array}{ccc}
v & \to & \eext{\Gamma_2} \\
h & \mapsto & \sigma(h)
\end{array}
\end{displaymath}
The mapping $\sigma$ will determine precisely how $\Gamma_2$ is inserted into $\Gamma_1$. Every external edge $e\in\eext{\Gamma_2}$ is incident to some unique internal vertex $v_e\in\vint{\Gamma_2}$.

We define a new graph
\[\Gamma:=\Gamma_1\circ_{v,\sigma}\Gamma_2\]
as follows. The edges of $\Gamma$ are formed by discarding the external edges of $\Gamma_2$,
\[ E(\Gamma):=E(\Gamma_1)\sqcup [E(\Gamma_2) - \eext{\Gamma_2}]. \]
The structure of the vertices of $\Gamma$ is largely unchanged except that when $\Gamma_2$ is inserted into $\Gamma_1$, the internal vertex $v$ of $\Gamma_1$ is disbanded. Hence we must describe the new vertices to which the incident half-edges of $v$ are now attached. If $u$ is either an internal vertex of $\Gamma_2$ or a vertex of $\Gamma_1$ other than $v$, we define a vertex $u'$ of $\Gamma$ by stipulating that for a half-edge $h$ of $\Gamma$
\[ h\in u' \Leftrightarrow h\in u \text{ or } (h\in v \text{ and } v_{\sigma(h)}=u).  \]
We can then make the definition
\[ V(\Gamma):=\{u':u\in\vint{\Gamma_2} \text{ or } u\in V(\Gamma_1)-\{v\} \}. \]
It only remains to specify the external vertices of $\Gamma$. These coincide with the external vertices of $\Gamma_1$,
\[ \vext{\Gamma}:=\vext{\Gamma_1}. \]

\begin{figure}[htp]
\centering
\includegraphics{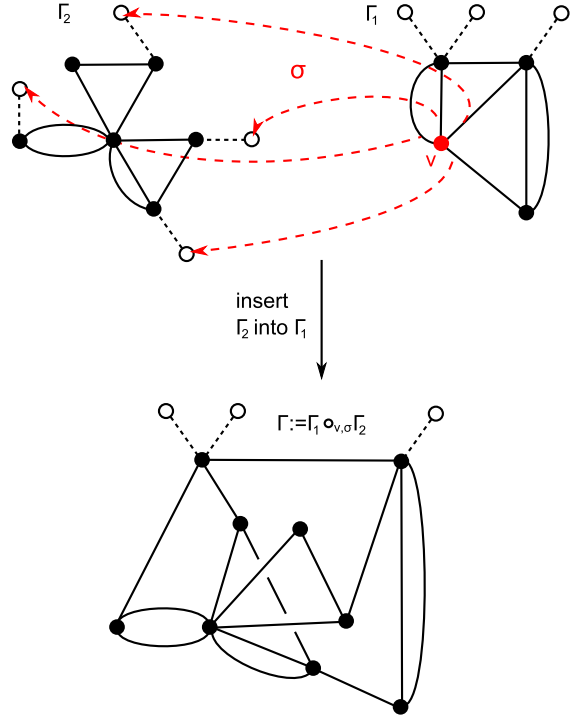}
\caption{Inserting graph $\Gamma_2$ into $\Gamma_1$ at vertex $v$ via the bijection $\sigma$.}
\end{figure}

With the above preliminaries out of the way, we make the definition
\begin{displaymath}
\begin{split}
\Gamma_1\circ\Gamma_2 &:= \sum_{v\in\vint{\Gamma_2}} \Gamma_1\circ_v\Gamma_2 \\
&:= \sum_{v\in\vint{\Gamma_2}} \left[\sum_{\sigma\in\Bij(\vext{\Gamma_1},v)} \Gamma_1\circ_{v,\sigma}\Gamma_2\right].
\end{split}
\end{displaymath}
\end{defi}

Connes and Kreimer observed that the above structure was pre-Lie in their paper \cite{CKelim}. Later, we will show that it is pre-Lie by comparing it to another well known pre-Lie structure on polynomial vector fields. The following theorem relates the above pre-Lie structure on connected graphs to the algebraic structure of the Hopf algebra $(\widehat{\mathcal{H}},\star,\Delta)$.

\begin{theorem}
Let $\Gamma_1,\Gamma_2\in\mathcal{H}_c^+$ be two connected graphs, then
\[ [\Gamma_1,\Gamma_2]=\Gamma_1\circ\Gamma_2-\Gamma_2\circ\Gamma_1. \]
\end{theorem}

\begin{proof}
This follows from the stronger identity
\[ \langle\Gamma_1\star\Gamma_2,\Gamma\rangle = \langle\Gamma_1\circ\Gamma_2,\Gamma\rangle \]
for connected graphs $\Gamma_1,\Gamma_2,\Gamma\in\mathcal{H}_c^+$ which was proven in Theorem 2 of \cite{CKhopf}. From this and equation \eqref{eqn_stardeform} it follows that
\[ \Gamma_1\star\Gamma_2 = \Gamma_1\sqcup\Gamma_2 + \Gamma_1\circ\Gamma_2 \]
for connected graphs $\Gamma_1,\Gamma_2\in\mathcal{H}_c^+$, from which the theorem follows.
\end{proof}

\section{Nondiagrammatic description of the Connes-Kreimer Hopf algebra} \label{sec_nongraphical}

In this section we prove our main theorem relating the insertion pre-Lie structure of Connes and Kreimer to the canonical pre-Lie structure on polynomial vector fields. Here, the invariant theory for the orthogonal groups plays a crucial role in allowing us to pass back and forth between invariant tensors and graphs. Since the objects that we work with are formal in nature, some technicalities arise in the definition of this pre-Lie structure as we must ensure that certain defining expressions are convergent; however, this is a relatively minor issue. We then use this pre-Lie structure to build a Hopf algebra isomorphic to the Connes-Kreimer Hopf algebra described in Section \ref{sec_graphical}.

\subsection{Invariant theory for the Orthogonal Group}

We begin our nondiagrammatic description of the Connes-Kreimer Hopf algebra by recalling some elementary facts concerning the invariant theory for the orthogonal groups. The invariant theory for $O(n)$ is described using \emph{chord diagrams}. These chord diagrams will later allow us to make contact with the graphical description of the Connes-Kreimer Hopf algebra outlined in Section \ref{sec_graphical}.

Recall that for every positive integer $n$, there is a canonical vector space $V_n$ of dimension $n$ equipped with the nondegenerate symmetric bilinear form \eqref{eqn_canonicalform}. Note that there are canonical inclusions
\begin{equation} \label{eqn_inclusion}
\begin{array}{ccc}
V_n & \to & V_{n+1} \\
x_i & \mapsto & x_i
\end{array}
\end{equation}
defined in terms of the basis elements $x_i$ which preserve the bilinear forms on $V_n$ and $V_{n+1}$.

Now we want to make some definitions concerning the invariant theory for the group $O(n)$ of linear endomorphisms of $V_n$ which preserve the canonical nondegenerate bilinear form.

\begin{defi}
A \emph{chord diagram} $c$ on the set $\{1,\ldots,2N\}$ is a partition of $\{1,\ldots,2N\}$ into pairs
\begin{equation} \label{eqn_chordrep}
c=\{i_1,j_1\},\{i_2,j_2\}\ldots,\{i_N,j_N\}.
\end{equation}
Note that the sets in this partition are not ordered, even though we have chosen an ordering in \eqref{eqn_chordrep} in order to represent it; hence both
\[ \{j_1,i_1\},\ldots,\{j_N,i_N\} \quad\text{and}\quad \{i_{\sigma(1)},j_{\sigma(1)}\},\ldots,\{i_{\sigma(N)},j_{\sigma(N)}\},\quad\sigma\in\mathbb{S}_n \]
represent the same chord diagram. The set of all chord diagrams will be denoted by $\chord{N}$.
\end{defi}

\begin{defi} \label{def_invariant}
Given a chord diagram $c\in\chord{N}$, define the $O(n)$-invariant $\beta_c\in\left[(V_n^*)^{\otimes 2N}\right]^{O(n)}$ by the formula
\[ \beta_c(x_1\otimes\cdots\otimes x_{2N})= \langle x_{i_1},x_{j_1} \rangle \cdots \langle x_{i_N},x_{j_N} \rangle. \]
It is clear that the invariant $\beta_c$ does not depend upon the choice of representation \eqref{eqn_chordrep} of the chord diagram. Similarly, we may define an $O(n)$-coinvariant $z_c\in\left[V_n^{\otimes 2N}\right]_{O(n)}$ if $n\geq N$ as follows. Define the permutation $\sigma_c$ by
\[ \sigma_c:=\left(\begin{array}{ccccccc} 1 & 2 & 3 & 4 & \ldots & 2N-1 & 2N \\ \downarrow & \downarrow & \downarrow & \downarrow & \ldots & \downarrow & \downarrow \\ i_1 & j_1 & i_2 & j_2 & \ldots & i_N & j_N \end{array}\right). \]
Define $z_c$ by
\[ z_c:=\sigma_c\cdot(x_1\otimes x_1\otimes x_2\otimes x_2 \otimes \cdots \otimes x_N \otimes x_N). \]
Note that even though the permutation $\sigma_c$ is not well-defined, depending as it does on a representation \eqref{eqn_chordrep} of the chord diagram $c$, the coinvariant $z_c$ does not depend on the choice of this representation.
\end{defi}

The following result, which is crucial to our nondiagrammatic description of the Connes-Kreimer Hopf algebra, may be found in \cite{loday}.

\begin{lemma} \label{lem_invariants}
The spaces of $O(n)$ invariants and coinvariants may be described as follows:
\begin{enumerate}
\item \label{item_invariant}
The space $\left[(V_n^*)^{\otimes 2N}\right]^{O(n)}$ is spanned by the invariants $\{\beta_c\}_{c\in\chord{N}}$ defined in Definition \ref{def_invariant}. If $n\geq N$ then the set $\{\beta_c\}_{c\in\chord{N}}$ forms a basis of $\left[(V_n^*)^{\otimes 2N}\right]^{O(n)}$ that is indexed by chord diagrams.
\item \label{item_coinvariant}
If $n\geq N$ then the set $\{z_c\}_{c\in\chord{N}}$ forms a basis of the set $\left[V_n^{\otimes 2N}\right]_{O(n)}$.
\item \label{item_invanish}
If $k=2N+1$ is odd, then both $\left[(V_n^*)^{\otimes k}\right]^{O(n)}$ and $\left[(V_n^*)^{\otimes k}\right]_{O(n)}$ vanish;
\[ \left[(V_n^*)^{\otimes k}\right]^{O(n)} = \{0\} = \left[(V_n^*)^{\otimes k}\right]_{O(n)}. \]
\end{enumerate}
\end{lemma}

\begin{proof}
Item \ref{item_invariant} is a restatement of Theorem 9.5.2 and Theorem 9.5.5 of \cite{loday}. Item \ref{item_coinvariant} follows from Item \ref{item_invariant} by observing that the space $\left[(V_n^*)^{\otimes 2N}\right]^{O(n)}$ is dual to the space $\left[V_n^{\otimes 2N}\right]_{O(n)}$ and that the following equation holds;
\begin{equation} \label{eqn_invpair}
\beta_c(z_{c'}) = \delta_{cc'}, \quad c,c'\in\chord{N}.
\end{equation}
This equation implies that $\{z_c\}_{c\in\chord{N}}$ is a dual basis to the basis $\{\beta_c\}_{c\in\chord{N}}$ of $\left[(V_n^*)^{\otimes 2N}\right]^{O(n)}$.

Item \ref{item_invanish} follows as a simple consequence of the fact that these tensors are invariant under the orthogonal transformation
\[x\mapsto -x.\]
\end{proof}

\begin{rem}
Note that the invariants $\beta_c$ are natural with respect to the inclusion \eqref{eqn_inclusion}. If we denote by
\[ \beta_c^n\in\left[(V_n^*)^{\otimes 2N}\right]^{O(n)} \quad\text{and}\quad \beta_c^{n+1}\in\left[(V_{n+1}^*)^{\otimes 2N}\right]^{O(n+1)} \]
the invariants corresponding to the chord diagram $c\in\chord{N}$ which live in the tensor powers of the vector spaces $V_n^*$ and $V_{n+1}^*$ respectively, then
\begin{equation} \label{eqn_invinc}
\beta_c^{n+1}\circ i = \beta_c^n
\end{equation}
where $i:V_n\to V_{n+1}$ is the inclusion \eqref{eqn_inclusion}.
\end{rem}

\subsection{Pre-Lie structure of polynomial vector fields}

Having given some of the necessary preliminary definitions, we now begin to define the pre-Lie algebra that we will work with. This pre-Lie algebra will be based on the well-known pre-Lie algebra of polynomial vector fields; let us recall how this is defined. If $Y$ is a (finite-dimensional) vector space then the space
\[ S(Y)\otimes Y^* \]
may be identified with the space of polynomial vector fields on $Y$, otherwise known as the space of derivations $\Der(S(Y))$ of the algebra $S(Y)$. If $\xi:=p(y)\frac{\partial}{\partial y_i}$ and $\eta:=q(y)\frac{\partial}{\partial y_j}$ are polynomial vector fields, then their pre-Lie bracket is defined by the formula
\begin{equation} \label{eqn_polyprelie}
\xi\circ\eta:=\left(q(y)\frac{\partial}{\partial y_j}[p(y)]\right)\frac{\partial}{\partial y_i}.
\end{equation}

Although our pre-Lie structure will be based on \eqref{eqn_polyprelie}, the object that we will construct will be a formal object and our vector space $Y$ will not be finite-dimensional. Consequently, in defining our pre-Lie algebra, we will have to deal with (amongst other things) issues regarding the convergence of the defining expressions. We now turn to these matters and the definition of our pre-Lie algebra.

\subsection{The same pre-Lie structure on formal objects}

We wish to extend this definition of the pre-Lie bracket of polynomial vector fields to a certain profinite vector space that is a formal version of that considered above. The object that we construct, along with its fundamental pre-Lie structure, will be the crucial device in describing the algebraic structures of Section \ref{sec_graphical} in a nondiagrammatic manner. To this end, we start by introducing polynomials and power series.

\begin{defi}
Let $V$ be a finite-dimensional vector space. We will denote the polynomial algebra and the power series algebra on $V$ by
\[ \mathfrak{h}[V]:=S(V^*)=\bigoplus_{n=0}^\infty (V^*)^{\otimes n}/\mathbb{S}_n\quad\text{and}\quad\widehat{\mathfrak{h}}[V]:=\widehat{S}(V^*)=\prod_{n=0}^\infty (V^*)^{\otimes n}/\mathbb{S}_n \]
respectively. We will denote by $\mathfrak{h}_{\geq k}[V]$ and $\widehat{\mathfrak{h}}_{\geq k}[V]$ those subspaces of $\mathfrak{h}[V]$ and $\widehat{\mathfrak{h}}[V]$ generated by polynomials and power series of order $\geq k$;
\[ \mathfrak{h}_{\geq k}[V]:=\bigoplus_{n=k}^\infty (V^*)^{\otimes n}/\mathbb{S}_n\quad\text{and}\quad\widehat{\mathfrak{h}}_{\geq k}[V]:=\prod_{n=k}^\infty (V^*)^{\otimes n}/\mathbb{S}_n. \]
\end{defi}

Now we may introduce the underlying vector space that we will define our algebraic structures on.

\begin{defi} \label{defi_prelie}
Let $V$ be a vector space with a nondegenerate symmetric bilinear form. We define a vector space $\mathfrak{l}[V]$ by
\[ \mathfrak{l}[V]:=\left[\widehat{S}(\widehat{\mathfrak{h}}_{\geq 1}[V])\cotimes\widehat{\mathfrak{h}}[V]\right]^{O(V)}. \]
We shall denote the vector space $\mathfrak{l}[V_n]$ associated to the canonical vector space $V_n$ by $\mathfrak{l}^n$. The canonical inclusions \eqref{eqn_inclusion} induce maps
\[ \mathfrak{l}^{n+1}\to\mathfrak{l}^n. \]
We shall denote by $\mathfrak{l}^\infty$ the inverse limit of the vector spaces $\mathfrak{l}^n$;
\begin{equation} \label{eqn_invlim}
\mathfrak{l}^{\infty}:=\invlim{n}{\mathfrak{l}^n}.
\end{equation}
\end{defi}

Later, we shall define a pre-Lie structure on a subspace of $\mathfrak{l}[V]$ using Equation \eqref{eqn_polyprelie}. In order to accommodate our arguments regarding the convergence of certain expressions, we will introduce a grading on $\mathfrak{l}[V]$. The space $\mathfrak{l}[V]$ has two gradings
\begin{equation} \label{eqn_liegrading}
\mathfrak{l}[V]=\prod_{n=0}^\infty \mathfrak{l}_{n\bullet}[V] \quad\text{and}\quad \mathfrak{l}[V]=\prod_{k=0}^\infty \mathfrak{l}_{\bullet k}[V].
\end{equation}
Here, $x\in\mathfrak{l}_{n\bullet}[V]$ if and only if it is a tensor in $V^*$ of degree $2n$; that is to say that it is represented by a tensor in $(V^*)^{\otimes 2n}$. The space $\mathfrak{l}_{\bullet k}[V]$ is defined by
\[ \mathfrak{l}_{\bullet k}[V] := \left[\widehat{S}(\widehat{h}_{\geq 1}[V])\cotimes [(V^*)^{\otimes k}]_{\mathbb{S}_k}\right]^{O(V)}. \]
These gradings combine to give $\mathfrak{l}[V]$ a bigrading
\[ \mathfrak{l}[V]=\prod_{n,k=0}^\infty \mathfrak{l}_{nk}[V] \]
where $\mathfrak{l}_{nk}[V] := \mathfrak{l}_{n\bullet}[V]\cap\mathfrak{l}_{\bullet k}[V]$.

Using this grading we make the definition
\[ \mathfrak{l}_+[V] := \prod_{n>k\geq 0} \mathfrak{l}_{nk}[V] \]
It is on this space that we shall define our pre-Lie structure. We denote the corresponding subspace of $\left[\widehat{S}(\widehat{\mathfrak{h}}_{\geq 1}[V])\cotimes\widehat{\mathfrak{h}}[V]\right]$ by $\left[\widehat{S}(\widehat{\mathfrak{h}}_{\geq 1}[V])\cotimes\widehat{\mathfrak{h}}[V]\right]_+$ so that
\[ \mathfrak{l}_+[V] = \left[\widehat{S}(\widehat{\mathfrak{h}}_{\geq 1}[V])\cotimes\widehat{\mathfrak{h}}[V]\right]_+^{O(V)} \]

To define the pre-Lie structure we define isomorphisms
\[ D:\widehat{\mathfrak{h}}[V]\to\widehat{\mathfrak{h}}[V^*] \]
and
\[ T:\widehat{\mathfrak{h}}[V^*]\to\mathfrak{h}[V]^* \]
by the following formulae;
\begin{displaymath}
\begin{split}
D(f_1\otimes\cdots\otimes f_k) &:= \langle f_1,- \rangle^{-1} \otimes\cdots\otimes \langle f_k,- \rangle^{-1} \\
T(x_1\otimes\cdots\otimes x_k) &:= \left[f_1\otimes\cdots\otimes f_k\mapsto\sum_{\sigma\in\mathbb{S}_k} x_{\sigma(1)}(f_1)\cdots x_{\sigma(k)}(f_k)\right]
\end{split}
\end{displaymath}
where $\innprod^{-1}$ is the inverse inner product on $V^*$. A pre-Lie structure on $\left[\widehat{S}(\widehat{\mathfrak{h}}_{\geq 1}[V])\cotimes\widehat{\mathfrak{h}}[V]\right]_+$ may then be specified as the unique pre-Lie structure fitting into the following commutative diagram.
\begin{displaymath}
\xymatrix{ \left[\widehat{S}(\widehat{\mathfrak{h}}_{\geq 1}[V])\cotimes\mathfrak{h}[V]^*\right]_+ \cotimes \left[\widehat{S}(\widehat{\mathfrak{h}}_{\geq 1}[V])\cotimes\mathfrak{h}[V]^*\right]_+ \ar[r]^-{\circ} & \left[\widehat{S}(\widehat{\mathfrak{h}}_{\geq 1}[V])\cotimes\mathfrak{h}[V]^*\right]_+ \\ \left[\widehat{S}(\widehat{\mathfrak{h}}_{\geq 1}[V])\cotimes\widehat{\mathfrak{h}}[V]\right]_+ \cotimes \left[\widehat{S}(\widehat{\mathfrak{h}}_{\geq 1}[V])\cotimes\widehat{\mathfrak{h}}[V]\right]_+ \ar[r]^-{\circ} \ar[u]^{(\id\cotimes TD)\cotimes(\id\cotimes TD)} & \left[\widehat{S}(\widehat{\mathfrak{h}}_{\geq 1}[V])\cotimes\widehat{\mathfrak{h}}[V]\right]_+ \ar[u]^{(\id\cotimes TD)} }
\end{displaymath}
where the pre-Lie structure on the top row is that given by composition of polynomial vector fields as in Equation \eqref{eqn_polyprelie}. Here we must be careful of course, since we work with formally complete objects and our space $Y=\widehat{\mathfrak{h}}[V]$ is no longer finite-dimensional, to check that this pre-Lie structure is well-defined and does not lead to divergent expressions. Let us denote by $\left[\widehat{S}(\widehat{\mathfrak{h}}_{\geq 1}[V])\cotimes\widehat{\mathfrak{h}}[V]\right]_{nk}$ the component subspace of $\left[\widehat{S}(\widehat{\mathfrak{h}}_{\geq 1}[V])\cotimes\widehat{\mathfrak{h}}[V]\right]$ that corresponds to the graded component $\mathfrak{l}_{nk}[V]$ so that
\[ \mathfrak{l}_{nk}[V] = \left[\widehat{S}(\widehat{\mathfrak{h}}_{\geq 1}[V])\cotimes\widehat{\mathfrak{h}}[V]\right]_{nk}^{O(V)} \]
One may check easily that
\[ \circ \left(\left[\widehat{S}(\widehat{\mathfrak{h}}_{\geq 1}[V])\cotimes\widehat{\mathfrak{h}}[V]\right]_{n_1k_1},\left[\widehat{S}(\widehat{\mathfrak{h}}_{\geq 1}[V])\cotimes\widehat{\mathfrak{h}}[V]\right]_{n_2k_2}\right) \subset \left[\widehat{S}(\widehat{\mathfrak{h}}_{\geq 1}[V])\cotimes\widehat{\mathfrak{h}}[V]\right]_{n_1+n_2-k_2,k_1} \]
Furthermore
\[ \circ \left(\left[\widehat{S}(\widehat{\mathfrak{h}}_{\geq 1}[V])\cotimes\widehat{\mathfrak{h}}[V]\right]_{n_1k_1},\left[\widehat{S}(\widehat{\mathfrak{h}}_{\geq 1}[V])\cotimes\widehat{\mathfrak{h}}[V]\right]_{n_2k_2}\right) = \{0\}, \quad\text{if }k_2>2n_1-k_1. \]
Hence it follows that
\begin{displaymath}
\begin{split}
\circ \left(\left[\widehat{S}(\widehat{\mathfrak{h}}_{\geq 1}[V])\cotimes\widehat{\mathfrak{h}}[V]\right]_{\geq N,\bullet},\left[\widehat{S}(\widehat{\mathfrak{h}}_{\geq 1}[V])\cotimes\widehat{\mathfrak{h}}[V]\right]_+\right) & \subset \left[\widehat{S}(\widehat{\mathfrak{h}}_{\geq 1}[V])\cotimes\widehat{\mathfrak{h}}[V]\right]_{\geq N,\bullet} \\
\circ \left(\left[\widehat{S}(\widehat{\mathfrak{h}}_{\geq 1}[V])\cotimes\widehat{\mathfrak{h}}[V]\right]_+,\left[\widehat{S}(\widehat{\mathfrak{h}}_{\geq 1}[V])\cotimes\widehat{\mathfrak{h}}[V]\right]_{+;\geq 2N,\bullet}\right) & \subset \left[\widehat{S}(\widehat{\mathfrak{h}}_{\geq 1}[V])\cotimes\widehat{\mathfrak{h}}[V]\right]_{\geq N,\bullet}
\end{split}
\end{displaymath}
Since there is no problem in defining $\circ$ in each bidegree, it follows that $\circ$ is a well-defined pre-Lie structure on $\left[\widehat{S}(\widehat{\mathfrak{h}}_{\geq 1}[V])\cotimes\widehat{\mathfrak{h}}[V]\right]_+$ whose defining expression is convergent. It is now simple to check that this pre-Lie structure preserves the space of $O(V)$-invariants and hence gives rise to a pre-Lie structure on the space $\mathfrak{l}_+[V]$.

\subsection{Hopf algebra structure of $\mathfrak{l}^\infty$} \label{sec_tensorhopf}

We shall now explain how $\mathfrak{l}^\infty$ has the structure of a commutative cocommutative Hopf algebra. Later, we shall relate this structure to the commutative cocommutative Hopf algebra $(\widehat{\mathcal{H}},\sqcup,\Delta)$.

\begin{defi}
Let us denote the canonical multiplication on the symmetric algebra $\widehat{S}(\widehat{\mathfrak{h}}_{\geq 1}[V])$ by $\mu_1$ and the canonical multiplication on the polynomial algebra $\widehat{\mathfrak{h}}[V]$ by $\mu_2$. This yields a commutative multiplication
\[ \mu:=\mu_1\cotimes\mu_2:\left(\widehat{S}(\widehat{\mathfrak{h}}_{\geq 1}[V])\cotimes\widehat{\mathfrak{h}}[V]\right)\cotimes\left(\widehat{S}(\widehat{\mathfrak{h}}_{\geq 1}[V])\cotimes\widehat{\mathfrak{h}}[V]\right) \to \left(\widehat{S}(\widehat{\mathfrak{h}}_{\geq 1}[V])\cotimes\widehat{\mathfrak{h}}[V]\right) \]
on the space $\widehat{S}(\widehat{\mathfrak{h}}_{\geq 1}[V])\cotimes\widehat{\mathfrak{h}}[V]$. One may check easily that this descends to $O(V)$-invariants and hence gives rise to a commutative product on $\mathfrak{l}[V]$. The commutative products on the $\mathfrak{l}^n:=\mathfrak{l}[V_n]$ assemble to give rise to a commutative product on $\mathfrak{l}^\infty$.
\end{defi}

\begin{defi}
To define a comultiplication on $\mathfrak{l}^\infty$, we first observe that
\begin{equation} \label{eqn_teninvsys}
\mathfrak{l}^\infty\cotimes\mathfrak{l}^\infty = \invlim{m,n}{\mathfrak{l}^m\cotimes\mathfrak{l}^n}.
\end{equation}
Now the symmetric algebra $\widehat{S}(\widehat{\mathfrak{h}}_{\geq 1}[V])$ has a canonical cocommutative comultiplication $\delta_1$ and the polynomial algebra $\widehat{\mathfrak{h}}[V]$ has a canonical cocommutative comultiplication $\delta_2$, which yields a cocommutative comultiplication
\[ \delta:=\delta_1\cotimes\delta_2:\left(\widehat{S}(\widehat{\mathfrak{h}}_{\geq 1}[V])\cotimes\widehat{\mathfrak{h}}[V]\right) \to \left(\widehat{S}(\widehat{\mathfrak{h}}_{\geq 1}[V])\cotimes\widehat{\mathfrak{h}}[V]\right)\cotimes\left(\widehat{S}(\widehat{\mathfrak{h}}_{\geq 1}[V])\cotimes\widehat{\mathfrak{h}}[V]\right). \]

Now we have canonical complementary inclusions
\[ \xymatrix{ V_m \ar[rr]^{x_i\mapsto x_i} && V_{m+n} && V_n \ar[ll]_{x_{m+i}\mapsfrom x_i}} \]
which give rise to corresponding projections
\begin{equation} \label{eqn_projections}
\xymatrix{ V_m^*  & V_{m+n}^* \ar[l]_{\pi_m} \ar[r]^{\pi_n} & V_n^* }
\end{equation}
and to a map
\[ \xymatrix{ \left[\widehat{S}(\widehat{\mathfrak{h}}_{\geq 1}[V_{m+n}])\cotimes\widehat{\mathfrak{h}}[V_{m+n}]\right]^{O(V_{m+n})} \ar[r]^-{\delta} & \left[\left(\widehat{S}(\widehat{\mathfrak{h}}_{\geq 1}[V_{m+n}])\cotimes\widehat{\mathfrak{h}}[V_{m+n}]\right)\cotimes \left(\widehat{S}(\widehat{\mathfrak{h}}_{\geq 1}[V_{m+n}])\cotimes\widehat{\mathfrak{h}}[V_{m+n}]\right)\right]^{O(V_{m+n})} \ar[d]^-{\pi_m\cotimes\pi_n} \\ & \left[\widehat{S}(\widehat{\mathfrak{h}}_{\geq 1}[V_{m}])\cotimes\widehat{\mathfrak{h}}[V_{m}]\right]^{O(V_m)}\cotimes \left[\widehat{S}(\widehat{\mathfrak{h}}_{\geq 1}[V_{n}])\cotimes\widehat{\mathfrak{h}}[V_{n}]\right]^{O(V_{n})} & \\ \mathfrak{l}^{m+n} \ar@{=}[uu] \ar[r]^{\delta} & \mathfrak{l}^m\cotimes\mathfrak{l}^n \ar@{=}[u] } \]

This gives rise to a series of maps
\[ \xymatrix{\mathfrak{l}^\infty \ar[r] & \mathfrak{l}^{m+n} \ar[r]^-{\delta} & \mathfrak{l}^{m}\cotimes\mathfrak{l}^{n} } \]
One can check that these maps are compatible with the inverse system \eqref{eqn_teninvsys}. Note that this calculation hinges crucially on the fact that we work with spaces of $O(V)$-invariants. This gives rise to a cocommutative comultiplication
\[ \delta:\mathfrak{l}^\infty\to\mathfrak{l}^\infty\cotimes\mathfrak{l}^\infty. \]
\end{defi}

One may easily check that the comultiplication $\delta$ is compatible with the multiplication $\mu$ so that $(\mathfrak{l}^\infty,\mu,\delta)$ forms a bialgebra. Furthermore, both $\mu$ and $\delta$ have bidegree zero in the bigrading on $\mathfrak{l}^\infty$ described by \eqref{eqn_liegrading}; that is to say that
\begin{equation} \label{eqn_bialbigrad}
\begin{split}
\mu(\mathfrak{l}^\infty_{n_1k_1},\mathfrak{l}^\infty_{n_2k_2}) & \subset \mathfrak{l}^\infty_{n_1+n_2,k_1+k_2} \\
\delta(\mathfrak{l}^\infty_{nk}) & \subset \bigoplus_{\begin{subarray}{c} n_1+n_2=n \\ k_1+k_2=k \end{subarray}} (\mathfrak{l}^\infty_{n_1k_1}\otimes\mathfrak{l}^\infty_{n_2k_2})
\end{split}
\end{equation}

\subsection{From graphs to tensors}

We are now ready to describe the passage from the graphical description of the Connes-Kreimer Hopf algebra outlined in Section \ref{sec_graphical} to the nongraphical description in terms of tensors. Here, the invariant theory of the orthogonal groups plays the crucial role. Its description in terms of chord diagrams allows us to connect graphs with tensors. We begin by describing a way of associating a graph to a chord diagram, given some extra information.

\begin{defi} \label{def_chordgraph}
Let $k_1,\ldots,k_m>0$ be a series of positive integers, $k_0\geq 0$ be a nonnegative integer and $c\in\chord{N}$ be a chord diagram
\[ c:=\{i_1,j_1\},\{i_2,j_2\},\ldots,\{i_N,j_N\} \]
such that $\sum_{i=0}^m k_i=2N$. We define a graph $\Gamma_{k_1,\ldots,k_m;k_0}(c)$ as follows. The graph has half-edges
\[ h_{11},\ldots,h_{1k_1};h_{21},\ldots,h_{2k_2};\ldots;h_{m1},\ldots,h_{mk_m};h_{01},\ldots,h_{0k_0}. \]
Writing them in this order from left to right allows us to identify each half-edge with an integer $1\leq i\leq 2N$,
\[ i\mapsto h_i; \]
for instance $h_{k_1+3}=h_{23}$ and $h_{k_1+k_2+1}=h_{31}$. From this we can use the chord diagram $c$ to describe the edges of $\Gamma:=\Gamma_{k_1,\ldots,k_m;k_0}(c)$;
\[ E(\Gamma) := \{\{h_{i_1},h_{j_1}\},\{h_{i_2},h_{j_2}\},\ldots,\{h_{i_N},h_{j_N}\}\}. \]
The external vertices of $\Gamma$ will be
\[ \vext{\Gamma} := \{\{h_{01}\},\{h_{02}\},\ldots,\{h_{0k_0}\}\} \]
whilst the internal vertices of $\Gamma$ will be
\[ \vint{\Gamma} := \{\{h_{11},\ldots,h_{1k_1}\},\{h_{21},\ldots,h_{2k_2}\},\ldots,\{h_{m1},\ldots,h_{mk_m}\}\}. \]
\end{defi}

Now we are in a position to construct a natural map
\begin{equation}
\Phi:\widehat{\mathcal{H}}\to\mathfrak{l}^\infty.
\end{equation}
This map has bidegree zero in the grading \eqref{eqn_graphgrade} and \eqref{eqn_liegrading}. That is to say that
\begin{equation} \label{eqn_mapgrade}
\Phi(\mathcal{H}_{n\bullet k})\subset\mathfrak{l}^\infty_{nk}.
\end{equation}
To do this, we will describe a family of maps
\[ \Phi^n:\widehat{\mathcal{H}}\to\mathfrak{l}^n \]
which are compatible with the inverse system \eqref{eqn_invlim} and hence give rise to the map $\Phi$.

If $\Gamma\in\widehat{\mathcal{H}}$ is a graph, then we may write $\Gamma=\Gamma_{k_1,\ldots,k_m;k_0}(c)$ for some choice of integers $k_0,k_1,\ldots,k_m$ and chord diagram $c$. Let $N:=\frac{1}{2}\sum_{i=0}^m k_i$ denote the number of edges of $\Gamma$ and define the map
\[ [-]_{k_1,\ldots,k_m;k_0}:(V^*)^{\otimes 2N}\to S^m(\mathfrak{h}_{\geq 1}[V])\otimes\mathfrak{h}[V] \]
by the commutative diagram
\[ \xymatrix{(V^*)^{\otimes 2N} \ar@{=}[d] \ar[r]^{[-]_{k_1,\ldots,k_m;k_0}} & S^m(\mathfrak{h}_{\geq 1}[V])\otimes\mathfrak{h}[V] \\ (V^*)^{\otimes k_1}\otimes\ldots\otimes (V^*)^{\otimes k_m}\otimes (V^*)^{\otimes k_0} \ar[d] \\ \left[(V^*)^{\otimes k_1}\right]_{\mathbb{S}_{k_1}}\otimes\ldots\otimes \left[(V^*)^{\otimes k_m}\right]_{\mathbb{S}_{k_m}}\otimes \left[(V^*)^{\otimes k_0}\right]_{\mathbb{S}_{k_0}} \ar[r] & \mathfrak{h}_{\geq 1}[V]^{\otimes m}\otimes\mathfrak{h}[V] \ar[uu] } \]
Now we define
\begin{equation} \label{eqn_mainmap}
\Phi^n(\Gamma):=[\beta_c]_{k_1,\ldots,k_m;k_0}\in \left[S^m(\mathfrak{h}_{\geq 1}[V_n])\otimes\mathfrak{h}[V_n]\right]^{O(V_n)} \subset\mathfrak{l}^n.
\end{equation}

We explain the graphical interpretation of this definition in Figure \ref{fig_tensorgraph} below.

\begin{figure}[htp]
\centering
\includegraphics{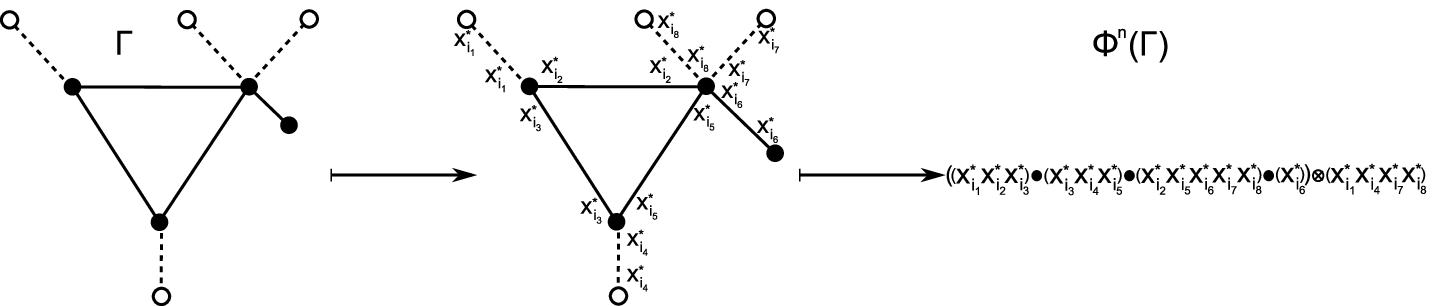}
\caption{The tensor associated to a graph by the map $\Phi^n$. The edges of the graph are decorated by the bilinear form
\[ \innprod=\sum_{i=1}^n x_i^* \otimes x_i^*. \]
The repeated $i_r$ indices indicate a summation $\sum_{i_r=1}^n$. Here, our tensor lies in $S^4(\mathfrak{h}_{\geq 1}[V_n])\otimes\mathfrak{h}[V_n]$.} \label{fig_tensorgraph}
\end{figure}

There of course may be many different equivalent choices of the integers $k_0,k_1,\ldots,k_m$ and chord diagram $c$, but one may easily observe that the map $\Phi^n$ does not depend upon this choice. By Equation \eqref{eqn_invinc}, it follows that the maps $\Phi^n$ are compatible with the inverse system \eqref{eqn_invlim} and hence yield the desired map \eqref{eqn_mainmap}.

Now, we relate the commutative cocommutative Hopf algebra structure on $\mathfrak{l}^\infty$ described in Section \ref{sec_tensorhopf} to the canonical commutative cocommutative Hopf algebra structure on $\widehat{\mathcal{H}}$ considered at the beginning of Section \ref{sec_ckhopf}.

\begin{theorem} \label{thm_bialgiso}
The map
\[ \Phi:(\widehat{\mathcal{H}},\sqcup,\Delta)\to(\mathfrak{l}^\infty,\mu,\delta) \]
is an isomorphism of bialgebras.
\end{theorem}

\begin{proof}
By Lemma \ref{lem_invariants}, item \eqref{item_invariant}; every tensor $\mathbf{f}\in\mathfrak{l}^n_{N\bullet}$ is a (finite) sum of elements of the form
\[ [\beta_c]_{k_1,\ldots,k_m;k_0} \]
where $c\in\chord{N}$ is a chord diagram and $\sum_{i=0}^m k_i=2N$. It follows that the map $\Phi^n:\widehat{\mathcal{H}}\to\mathfrak{l}^n$ is surjective for all $n$.

We begin by defining a family of maps $\Psi^n:\mathfrak{l}^n\to\widehat{\mathcal{H}}$ as follows. Given $\mathbf{f}\in\mathfrak{l}^n_{N\bullet}$ we may assume, by the preceding observation, that in defining $\Psi^n(\mathbf{f})$ we have
\[ \mathbf{f}=[\tilde{\mathbf{f}}]_{k_1,\ldots,k_m;k_0},\quad\tilde{\mathbf{f}}\in\left[(V_n^*)^{\otimes 2N}\right]^{O(V_n)}. \]
We then define $\Psi^n$ by
\begin{equation} \label{eqn_inverse}
\Psi^n(\mathbf{f}) = \left\{\begin{array}{rl}\sum_{c\in\chord{N}}\tilde{\mathbf{f}}(z_c)\Gamma_{k_1,\ldots,k_m;k_0}(c), & \text{if }  N\leq n \\ 0, & \text{otherwise} \end{array}\right\}.
\end{equation}
The map $\Psi^n$ is well-defined as if $\tilde{\mathbf{g}}\in\left[(V_n^*)^{\otimes 2N}\right]^{O(V_n)}$ is a another representative for $\mathbf{f}$
\[ [\tilde{\mathbf{g}}]_{k'_1,\ldots,k'_m;k'_0}=\mathbf{f}=[\tilde{\mathbf{f}}]_{k_1,\ldots,k_m;k_0} \]
we must have $k'_0=k_0$, $k'_i=k_{\tau(i)}$ for some permutation $\tau\in\mathbb{S}_m$ and $\tilde{\mathbf{g}}$ differs from $\tilde{\mathbf{f}}$ by some permutation $\sigma$ corresponding to the symmetries of the tensor $\mathbf{f}$. However, this symmetry $\sigma$ corresponds precisely to a symmetry of the terms $\Gamma_{k_1,\ldots,k_m;k_0}(c)$ which appear in the sum \eqref{eqn_inverse}. Hence the definition of $\Psi^n(\mathbf{f})$ does not depend upon this choice of representative.

We may now define a map $\Psi:\mathfrak{l}^\infty\to\widehat{\mathcal{H}}$ as follows. If $\pi_n:\mathfrak{l}^\infty\to\mathfrak{l}^n$ denotes the canonical projection from the inverse limit onto $\mathfrak{l}^n$, we define $\Psi$ by the formula
\[ \Psi(\mathbf{f}) := \lim_{n\to\infty} \Psi^n\pi_n(\mathbf{f}). \]
Note that it follows from the definition of $\Psi^n$ and Equation \eqref{eqn_invinc} that the sequence $\Psi^n\pi_n(\mathbf{f})$ is Cauchy and hence convergent.

It follows from Equation \eqref{eqn_invpair} that $\Psi\circ\Phi=\id$. Since, as we have already observed, the map $\Phi^n:\widehat{\mathcal{H}}\to\mathfrak{l}^n$ is surjective for all $n$, it follows that $\Phi\circ\Psi=\id$. Hence we have established that $\Phi$ is bijective.

To prove that $\Phi$ preserves the bialgebra structure, let us begin by establishing that $\Phi$ preserves multiplication. Suppose that
\[ \Gamma_1=\Gamma_{k_1,\ldots,k_{m_1};k_0}(c_1) \quad\text{and}\quad \Gamma_2=\Gamma_{k'_1,\ldots,k'_{m_2};k'_0}(c_2) \]
are two graphs with $N_1$ and $N_2$ edges respectively. We shall define a chord diagram $c_1\sqcup c_2$ such that
\begin{equation} \label{eqn_chordunion}
\Gamma_1\sqcup\Gamma_2=\Gamma_{k_1,\ldots,k_{m_1},k'_1,\ldots,k'_{m_2};k_0+k'_0}(c_1\sqcup c_2).
\end{equation}
We may write the chord diagrams $c_1$ and $c_2$ as
\[ c_1=\{i_1,j_1\},\ldots,\{i_{N_1},j_{N_1}\} \quad\text{and}\quad c_2=\{i'_1,j'_1\},\ldots,\{i'_{N_2},j'_{N_2}\}. \]
Define
\[ K:=\sum_{i=1}^{m_1} k_i \quad\text{and}\quad K':=\sum_{i=1}^{m_2} k'_i \]
and define the permutation $\sigma\in\mathbb{S}_{2N_1+2N_2}$ as the $K'$-fold $(k_0+K')$-cycle
\[ \sigma:=(K+1,\cdots,2N_1,2N_1+1,\cdots,2N_1+K')^{K'} \]
which swaps the blocks $\{K+1,\ldots,2N_1\}$ and $\{2N_1+1,\ldots,2N_1+K'\}$. Now define $c_1\sqcup c_2\in\chord{N_1+N_2}$ to be the image under the permutation $\sigma$ of the chord diagram
\[ \{i_1,j_1\},\ldots,\{i_{N_1},j_{N_1}\},\{2N_1+i'_1,2N_1+j'_1\},\ldots,\{2N_1+i'_{N_2},2N_1+j'_{N_2}\} \]
so that \eqref{eqn_chordunion} holds. From this we see that
\begin{displaymath}
\begin{split}
\Phi^n(\Gamma_1\sqcup\Gamma_2) &= [\beta_{c_1\sqcup c_2}]_{k_1,\ldots,k_{m_1},k'_1,\ldots,k'_{m_2};k_0+k'_0} \\
&= [\sigma\cdot(\beta_{c_1}\otimes\beta_{c_2})]_{k_1,\ldots,k_{m_1},k'_1,\ldots,k'_{m_2};k_0+k'_0} \\
&= \mu([\beta_{c_1}]_{k_1,\ldots,k_{m_1};k_0}, [\beta_{c_2}]_{k'_1,\ldots,k'_{m_2};k'_0}) \\
&= \mu(\Phi^n(\Gamma_1),\Phi^n(\Gamma_2)).
\end{split}
\end{displaymath}

Finally, we must show that $\Phi$ preserves the comultiplications. Since $\widehat{\mathcal{H}}$ is (formally) generated by primitive elements, i.e. connected graphs, it is sufficient to show that $\Phi$ maps primitive elements to primitive elements. Consider the vector space $V_{m+n}$ and let $\{x_i^*\}_{i=1}^{m+n}$ be the basis of $V_{m+n}^*$ which is dual to the basis $\{x_i\}_{i=1}^{m+n}$ of $V_{m+n}$. We may write the bilinear form \eqref{eqn_canonicalform} on $V_{m+n}$ as
\[ \innprod = \sum_{i=1}^{m+n} x_i^*\otimes x_i^*. \]
Consider the projections
\[ \pi_m:V_{m+n}^*\to V_m^* \quad\text{and}\quad \pi_n:V_{m+n}^*\to V_n^* \]
defined in Figure \eqref{eqn_projections}. Note that
\begin{align}
\label{eqn_projvanishone} \pi_n(x_i^*)=0, & \quad 1\leq i\leq m & \pi_n(x_{m+i}^*)=x_i^*, & \quad 1\leq i\leq n \\
\label{eqn_projvanishtwo} \pi_m(x_i^*)=x_i^*, & \quad 1\leq i\leq m & \pi_m(x_{m+i}^*)=0, & \quad 1\leq i\leq n
\end{align}
To prove that $\Phi$ maps primitive elements to primitive elements, we must show that
\begin{equation} \label{eqn_primtoprim}
\delta\Phi^{m+n}(\Gamma) = \Phi^m(\Gamma)\otimes\mathbf{1}+\mathbf{1}\otimes\Phi^n(\Gamma)
\end{equation}
for all connected graphs $\Gamma$. Now
\[ \delta\Phi^{m+n}(\Gamma) = (\pi_m\cotimes\pi_n)\left[\Phi^{m+n}(\Gamma)\otimes\mathbf{1} + \mathbf{1}\otimes\Phi^{m+n}(\Gamma) + \sum_j Y_j\otimes Y'_j\right], \]
where $\sum_j Y_j\otimes Y'_j$ is a (finite) sum of tensors where both $Y_j$ and $Y'_j$ lie in $\mathfrak{l}^{m+n}_{\geq 1,\bullet}$. Since the graph $\Gamma$ is connected, this sum will be a sum of terms of the form
\begin{displaymath}
\begin{split}
\sum_{i=1}^m\left[a_1\cdots a_{p_1}\cdot x_i^*\cdot b_1\cdots b_{q_1}\right]_{k_1,\ldots,k_{l_1};k_0} & \otimes\left[a'_1\cdots a'_{p_2}\cdot x_i^*\cdot b'_1\cdots b'_{q_2}\right]_{k'_1,\ldots,k'_{l_2};k'_0} \\
& + \\
\sum_{j=1}^n\left[a_1\cdots a_{p_1}\cdot x_{m+j}^*\cdot b_1\cdots b_{q_1}\right]_{k_1,\ldots,k_{l_1};k_0} & \otimes\left[a'_1\cdots a'_{p_2}\cdot x_{m+j}^*\cdot b'_1\cdots b'_{q_2}\right]_{k'_1,\ldots,k'_{l_2};k'_0}
\end{split}
\end{displaymath}
where $a_i,a'_i,b_i,b'_i\in V_{m+n}^*$. It follows from \eqref{eqn_projvanishone} that the first sum vanishes after applying $(\pi_m\cotimes\pi_n)$, and from \eqref{eqn_projvanishtwo} that the second sum vanishes after applying $(\pi_m\cotimes\pi_n)$. Hence
\[ (\pi_m\cotimes\pi_n)\left[\sum_j Y_j\otimes Y'_j\right]=0, \]
from which \eqref{eqn_primtoprim} follows.
\end{proof}

\subsection{The main theorem}

Now we may begin to compare the primitive elements of these two Hopf algebras. Let us denote the primitive elements of the bialgebra $\mathfrak{l}^\infty$ by $\mathcal{P}\mathfrak{l}^\infty$. Since by Remark \ref{rem_connprim} the space of primitive elements of $\widehat{\mathcal{H}}$ is $\widehat{\mathcal{H}}_c$, it follows from Theorem \ref{thm_bialgiso} that $\Phi$ induces an isomorphism
\[ \Phi:\widehat{\mathcal{H}}_c\to\mathcal{P}\mathfrak{l}^\infty. \]

\begin{prop} \label{prop_mainiso}
$\Phi$ induces an isomorphism
\[ \Phi:\widehat{\mathcal{H}}_c^+ \to \mathcal{P}\mathfrak{l}^\infty_+ \]
where $\mathcal{P}\mathfrak{l}^\infty_+:=\mathcal{P}\mathfrak{l}^\infty\cap\mathfrak{l}^\infty_+$.
\end{prop}

\begin{proof}
By Equation \eqref{eqn_bialbigrad} it follows that $\mathcal{P}\mathfrak{l}^\infty_+$ is a bigraded subspace of $\mathfrak{l}^\infty_+$. Since by Equation \eqref{eqn_mapgrade}, the map $\Phi$ has bidegree zero, it follows that $\Phi$ induces an isomorphism
\[ \Phi:\prod_{n>k} \mathcal{H}_{c;n\bullet k} \to \mathcal{P}\mathfrak{l}^\infty_+. \]
Now observe that a connected graph has no internal edges if and only if the number $n$ of its edges does not exceed the number $k$ of its external vertices. From this it follows that $\widehat{\mathcal{H}}_c^+=\prod_{n>k} \mathcal{H}_{c;n\bullet k}$.
\end{proof}

We are nearly ready to formulate our main theorem. Having defined pre-Lie structures on the spaces $\mathfrak{l}^n_+$, we require the following lemma to extend this structure to $\mathfrak{l}^\infty_+$.

\begin{lemma}
The following diagram commutes;
\[ \xymatrix{ \mathfrak{l}_+^{n+1}\cotimes\mathfrak{l}_+^{n+1} \ar[r]^\circ \ar[d] & \mathfrak{l}_+^{n+1} \ar[d] \\ \mathfrak{l}_+^n\cotimes\mathfrak{l}_+^n \ar[r]^\circ & \mathfrak{l}_+^n } \]
where the vertical maps are the canonical projections.
\end{lemma}

\begin{proof}
This follows from Equation \eqref{eqn_prelieidentity} which is proved in the course of the next theorem and the fact that the map $\Phi^n:\widehat{\mathcal{H}}\to\mathfrak{l}^n$ is surjective.
\end{proof}

Hence these pre-Lie structures give rise to a map
\[ \circ:\mathfrak{l}^\infty_+\cotimes\mathfrak{l}^\infty_+\to\mathfrak{l}^\infty_+ \]
describing a pre-Lie structure on $\mathfrak{l}^\infty_+$. Our main theorem that follows implies that $\mathcal{P}\mathfrak{l}^\infty_+$ is a pre-Lie subalgebra of $\mathfrak{l}^\infty_+$.

\begin{theorem} \label{thm_prelieiso}
The map
\[ \Phi:\widehat{\mathcal{H}}_c^+\to\mathcal{P}\mathfrak{l}^\infty_+ \]
is an isomorphism of pre-Lie algebras.
\end{theorem}

\begin{proof}
Let $x_1^*,\ldots, x_k^*$ denote the basis of $V_k^*$ which is dual to the basis $x_1,\ldots, x_k$ of $V_k$. Given a linear functional $\alpha\in V_k^*$, we have the identity
\begin{equation} \label{eqn_invinnprod}
\alpha = [\innprod\otimes\alpha]\circ[\id\otimes\innprod^{-1}] = \sum_{i=1}^k \langle \alpha,x_i^* \rangle^{-1} x_i^*.
\end{equation}

By Proposition \ref{prop_mainiso}, we need only prove that the map $\Phi$ preserves the pre-Lie structures. Let
\[ \Gamma_1:=\Gamma_{k_1,\ldots,k_n;k_0}(c_1) \quad\text{and}\quad \Gamma_2:=\Gamma_{l_1,\ldots,l_m;l_0}(c_2) \]
be two connected graphs. We may write the half-edges of $\Gamma_1$ as
\begin{equation} \label{eqn_halfedges}
h_{11},\ldots,h_{1k_1};h_{21},\ldots,h_{2k_2};\ldots;h_{n1},\ldots,h_{nk_n};h_{01},\ldots,h_{0k_0}.
\end{equation}
The $i$th internal vertex will be denoted by $v_i:=\{h_{i1},\ldots,h_{ik_i}\}$ and the $i$th external vertex will be denoted by $u_i:=\{h_{0i}\}$. Likewise, we denote by $h'_{ij}$, $v'_i$ and $u'_i$ the corresponding subobjects of $\Gamma_2$.

Every external vertex $u'_i$ of $\Gamma_2$ is attached through some external edge $e_i$ to some internal vertex $v'_{p(i)}$ of $\Gamma_2$. We may write this external edge as
\[ e_i=\{h'_{0i},h'_{p(i)q(i)}\} \]
where the half-edge $h'_{p(i)q(i)}$ is incident to $v'_{p(i)}$. Let us denote by $r(i)$ the position of this half-edge from the left as it appears in \eqref{eqn_halfedges}; for instance, if $h'_{p(i)q(i)}=h'_{12}$ then $r(i)=2$ and if $h'_{p(i)q(i)}=h'_{37}$ then $r(i)=l_1+l_2+7$. We may assume that
\[ r(1) < r(2) < \ldots < r(l_0). \]

We may describe the pre-Lie product of $\Gamma_1$ and $\Gamma_2$ by
\[ \Gamma_1\circ\Gamma_2 = \sum_{\begin{subarray}{c} 1\leq i\leq n: \\ l_0=k_i \end{subarray}}\sum_{\sigma\in\mathbb{S}_{l_0}} \Gamma_1\circ_{v_i,\sigma}\Gamma_2. \]
Consider the graph whose half-edges and vertices may be listed as
\begin{displaymath}
\begin{array}{rl}
\text{internal vertices} & \left\{\begin{array}{l} h'_{11},\ldots,h'_{1l_1}; \\ \qquad\vdots \\ h'_{m1},\ldots,h'_{ml_m}; \\ h_{11},\ldots,h_{1k_1}; \\ \qquad\vdots \\ \widehat{(h_{i1},\ldots,h_{ik_i})} \\ \qquad\vdots \\ h_{n1},\ldots,h_{nk_n};  \end{array}\right. \\
\text{external vertices} & \left\{\begin{array}{l} h_{01},\ldots,h_{0k_0} \end{array}\right.
\end{array}
\end{displaymath}
The graph $\Gamma_1\circ_{v_i,\sigma}\Gamma_2$ is obtained from this graph by replacing each half-edge $h'_{p(j),q(j)}$ of $\Gamma_2$ with the half-edge $h_{i,\sigma^{-1}(j)}$ from $\Gamma_1$. The edge structure of $\Gamma_1\circ_{v_i,\sigma}\Gamma_2$ is inherited from $\Gamma_1$ and $\Gamma_2$, that is to say that
\[ E(\Gamma_1\circ_{v_i,\sigma}\Gamma_2)\subset E(\Gamma_1) \cup E(\Gamma_2). \]

By definition, we have
\[ \Phi^k(\Gamma_1) = [\beta_{c_1}]_{k_1,\ldots,k_n;k_0} \quad\text{and}\quad \Phi^k(\Gamma_2) = [\beta_{c_2}]_{l_1,\ldots,l_m;l_0}. \]
We may write $\beta_{c_2}$ as a \emph{sum} of terms of the form
\[ \sum_{i_1,i_2,\ldots,i_{l_0}=1}^k (a_1 \cdots a_{r(1)-1}) x_{i_1}^* (a_{r(1)+1} \cdots a_{r(2)-1}) x_{i_2}^* a_{r(2)+1} \cdots a_{r(l_0)-1} x_{i_{l_0}}^* (a_{r(l_0)+1} \cdots a_{\sum_{j=1}^m l_j}) \cdot (x_{i_1}^* x_{i_2}^* \cdots x_{i_{l_0}}^*) \]
where each $a_i\in V_k^*$. Note that the left-most tensors of the form $x_{i_r}^*$ appear in the same location that the half-edges $h'_{p(r),q(r)}$ appear in \eqref{eqn_halfedges}.

If we write $\beta_{c_1}$ as a sum of tensors of the form
\[ (b_{11}\cdots b_{1k_1})\cdots (b_{i1}\cdots b_{ik_i})\cdots (b_{n1}\cdots b_{nk_n})\cdot (b_{01}\cdots b_{0k_0}) \]
where $b_{ij}\in V_k^*$, then we may write $\Phi^k(\Gamma_1\circ_{v_i,\sigma}\Gamma_2)$ as the corresponding sum of tensors of the form
\begin{equation} \label{eqn_tensorcompose}
\left[\mathbf{a}_1 b_{i\sigma^{-1}(1)} \mathbf{a}_2 b_{i\sigma^{-1}(2)} \mathbf{a}_3 \cdots \mathbf{a}_{l_0} b_{i\sigma^{-1}(l_0)} \mathbf{a}_{l_0+1} \cdot \mathbf{b}_1 \cdots \widehat{\mathbf{b}_i}\cdots \mathbf{b}_n\cdot \mathbf{b}_0\right]_{l_1,\ldots,l_m,k_1,\ldots,\hat{k_i},\ldots,k_n;k_0}
\end{equation}
where we have introduced the following shorthand for the above strings of tensors:
\begin{displaymath}
\begin{split}
\mathbf{a}_1 &:= (a_1 \cdots a_{r(1)-1}) \\
\mathbf{a}_i &:= (a_{r(i-1)+1} \cdots a_{r(i)-1}), \quad 2\leq i \leq l_0 \\
\mathbf{a}_{l_0+1} &:= (a_{r(l_0)+1} \cdots a_{\sum_{j=1}^m l_j}) \\
\mathbf{b}_i &:= (b_{i1}\cdots b_{ik_i})
\end{split}
\end{displaymath}

Now we would like to compute
\[ \Phi^k(\Gamma_1)\circ\Phi^k(\Gamma_2) = [\beta_{c_1}]_{k_1,\ldots,k_n;k_0} \circ [\beta_{c_2}]_{l_1,\ldots,l_m;l_0}. \]
By definition this is
\[ \Phi^k(\Gamma_1)\circ\Phi^k(\Gamma_2) = \sum_{\begin{subarray}{c} 1\leq i\leq n: \\ l_0=k_i \end{subarray}}\sum_{\sigma\in\mathbb{S}_{l_0}} [\beta_{c_1}\circ_{i,\sigma}\beta_{c_2}]_{l_1,\ldots,l_m,k_1,\ldots,\hat{k_i},\ldots,k_n;k_0} \]
where
\begin{multline*}
\beta_{c_1}\circ_{i,\sigma}\beta_{c_2}:= \\ \sum_{i_1,i_2,\ldots,i_{l_0}=1}^k \langle x_{i_1}^*,b_{i\sigma^{-1}(1)} \rangle^{-1} \cdots \langle x_{i_{l_0}}^* , b_{i\sigma^{-1}(l_0)} \rangle^{-1} \mathbf{a}_1 x_{i_1}^*\mathbf{a}_2 x_{i_2}^* \mathbf{a}_3\cdots\mathbf{a}_{l_0}x_{i_{l_0}}^*\mathbf{a}_{l_0+1} \mathbf{b}_1\cdots\widehat{\mathbf{b}_i}\cdots\mathbf{b}_n\cdot\mathbf{b}_0
\end{multline*}
By Equation \eqref{eqn_invinnprod} this is equal to \eqref{eqn_tensorcompose}, hence we have shown that
\begin{equation} \label{eqn_prelieidentity}
\Phi^k(\Gamma_1\circ\Gamma_2) = \Phi^k(\Gamma_1)\circ\Phi^k(\Gamma_2)
\end{equation}
from which the theorem follows. The following picture should assist the reader in visualizing the details of the argument.

\begin{figure}[htp]
\centering
\includegraphics{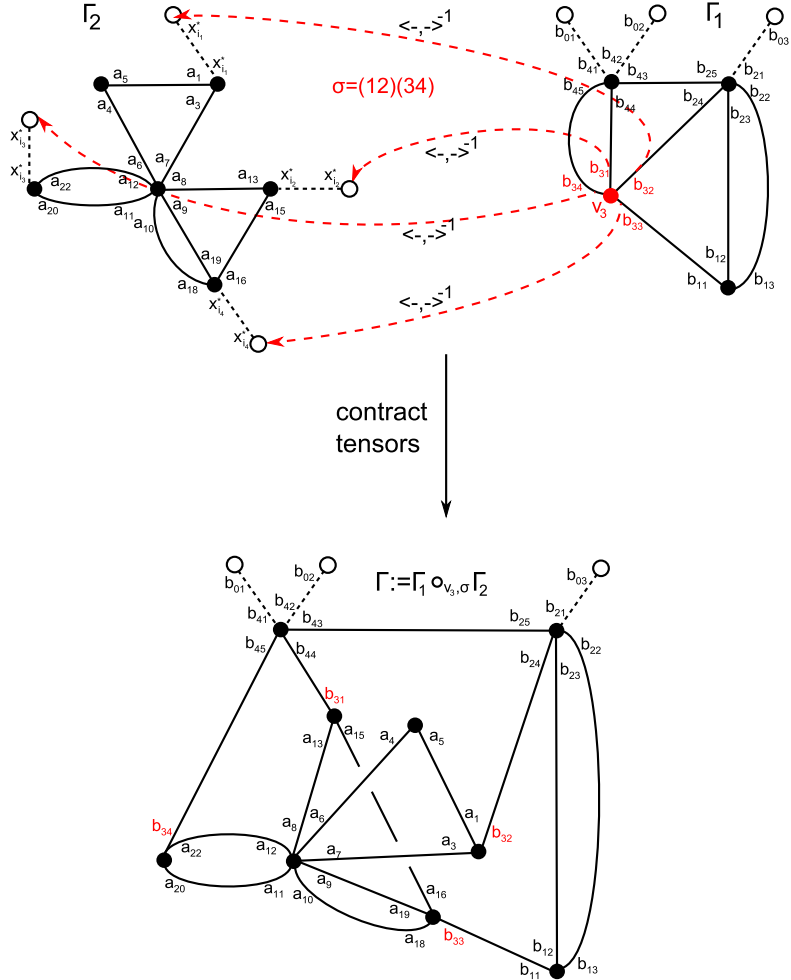}
\caption{Contracting tensors according to the pre-Lie structure on $\mathfrak{l}^n$ corresponds to inserting one graph into another.}
\end{figure}
\end{proof}

This yields the following nongraphical description of the Connes-Kreimer Hopf algebra.

\begin{cor}
There is a natural isomorphism of Hopf algebras
\begin{displaymath}
\begin{array}{ccc}
\widehat{\mathcal{U}}(\mathcal{P}\mathfrak{l}^\infty_+) & \to & \widehat{\mathcal{H}}^+, \\
x_1\cotimes\cdots\cotimes x_k & \mapsto & \Phi^{-1}(x_1)\star\cdots\star\Phi^{-1}(x_k).
\end{array}
\end{displaymath}
\end{cor}

\begin{proof}
This follows from Theorem \ref{thm_prelieiso} and Theorem \ref{thm_milmor}.
\end{proof}

\end{document}